\documentclass[11pt,reqno]{amsart}

\usepackage{marginnote,amssymb,mathrsfs,latexsym,verbatim,pdfsync,color,mathabx}

\newcommand{\bbC}{{\mathbb C}}

\newcommand{\bbN}{{\mathbb N}}
\newcommand{\bbR}{{\mathbb R}}

\newcommand{\bbZ}{{\mathbb Z}}

\def\bP{{\mathbf P}}

\def\cB{{\mathcal B}}
\def\cC{{\mathcal C}}
\def\cD{{\mathcal D}}
\def\cE{{\mathcal E}}

\def\cF{{\mathcal F}}

\def\cI{{\mathcal I}}

\def\cP{{\mathcal P}}

\def\cS{{\mathcal S}}
\def\cT{{\mathcal T}}

\def\Im{\operatorname{Im}}

\def\la{\langle}
\def\ra{\rangle}

\def\eps{\varepsilon}
\def\vp{\varphi}

\def\ms{\medskip}

\def\endpf{\medskip\hfill $\Box$

\ms

}  
\def\epf{\endpf}

\def\Hol{\operatorname{Hol}}

\def\supp{\operatorname{supp}}

\def\sinc{\operatorname{sinc}}

\def\Bsp{\cB^{s,p}_a}
\def\Lps{\dot{L}^p_s}

\def\BMO{\operatorname{BMO}}

\def\vp{\varphi}

\newtheorem{thm}{Theorem}[section]
\newtheorem{prop}[thm]{Proposition}
\newtheorem{cor}[thm]{Corollary}
\newtheorem{lem}[thm]{Lemma}
\newtheorem{defn}[thm]{Definition}
\newtheorem{remark}[thm]{Remark}

\newtheorem{theorem}{Theorem}
\newtheorem*{Theorem}{Theorem}

\begin{document}

\title[Fractional Bernstein spaces]{Fractional Paley--Wiener and Bernstein spaces} 
\author[A. Monguzzi, M. M. Peloso, M. Salvatori]{Alessandro Monguzzi*, 
Marco M. Peloso**, Maura Salvatori**}

 \address{*Dipartimento di Matematica e Applicazioni, Universit\`a
   degli Studi di Milano--Bicocca, Via R. Cozzi 55, 20126 Milano,
   Italy}   
\email{{\tt alessandro.monguzzi@unimib.it}}
\address{**Dipartimento di Matematica, Universit\`a degli Studi di
  Milano, Via C. Saldini 50, 20133 Milano, Italy}
\email{{\tt marco.peloso@unimi.it}}
\email{{\tt maura.salvatori@unimi.it}}

\thanks{{\em Math Subject Classification} 30D15, 26A33 }
\thanks{Authors were partially supported by the 2015 PRIN grant
  \emph{Real and Complex Manifolds: Geometry, Topology and Harmonic analysis}  
  of the Italian Ministry of Education (MIUR), and are members of the
  Gruppo Nazionale per l'Analisi Matematica, la Probabilit\`a e le
  loro 
Applicazioni (GNAMPA) of the Istituto Nazionale di Alta Matematica (INdAM) }

\begin{abstract}
We introduce and study a family of spaces of entire functions in one
variable that generalise 
the classical Paley--Wiener and Bernstein spaces. Namely, we consider
entire functions of exponential type $a$ whose restriction to the real
line belongs to the homogeneous Sobolev space $\dot{W}^{s,p}$  and we
call these spaces fractional Paley--Wiener if $p=2$ and fractional Bernstein
spaces if $p\in(1,\infty)$, that we denote by $PW^s_a$ and $\Bsp$,
respectively. For these spaces we provide a Paley--Wiener type
characterization, we remark some facts about the sampling problem in
the Hilbert setting and prove generalizations of the classical 
Bernstein and Plancherel--P\'olya inequalities.  We conclude by
discussing a number 
of open questions.
\end{abstract}
\maketitle
\begin{center}
{\emph{Dedicated to the memory of Elias M.\ Stein}}
\end{center}

\section{Introduction and statement of the main results}

A renowned theorem due to R. Paley and N. Wiener \cite{PW} characterizes the
entire functions of exponential type $a>0$ whose restriction to the real
line is square-integrable in terms of the support of the Fourier
transform of their restriction to the real line. An analogous
characterization holds for entire functions of exponential type $a$
whose restriction to to the real line belongs to some $L^p$ space,
$p\neq 2$  \cite{Bernstein}.
To be precise, let $\cE_a$
be the space of entire functions of exponential type $a$, 
\begin{equation}\label{exp-type}
\cE_a = \left\{ f\in\Hol(\bbC): \, \text{for every } \eps>0\
\text{there exists\ } C_\eps>0  \text{ such that \ } |f(z)|\le C_\eps
e^{(a+\eps)|z|} \right\} \,.
\end{equation}
Then, for any $p\in(1,\infty)$, the Bernstein space $\cB^p_a$ is defined as
$$
\cB^p_{a}=\left\{ f\in \cE_a: f_0\in L^p, \, \|f\|_{\cB^p_a} = \|f_0\|_{L^p} \right\}
$$
where $f_0:=f|_{\bbR}$ denotes the restriction of $f$ to the real line
and $L^p$ is the standard Lebesgue space.
In the Hilbert setting $p=2$, the Bernstein space $\cB^2_a$ is more
commonly known as the Paley--Wiener space and we will denote it by $PW_a$
in place of $\cB^2_a$.

Let 
$\cS$  and
$\cS'$ denote the space of Schwartz functions and
the space of tempered
distributions, resp.
For $f\in\cS$ we
 equivalently denote by $\widehat f$ or $\cF f$ the Fourier
 transform given by
 \begin{equation*}
   \widehat f(\xi)
   =  \frac{1}{\sqrt{2\pi}} \int_{\bbR} f (x)e^{-ix\xi}\, dx. 
\end{equation*} 
The Fourier transform $\cF$ is an isomorphism of $\cS$
onto itself with inverse given by
$$
\cF^{-1}f(x) =
\frac{1}{\sqrt{2\pi}}\int_{\bbR} \widehat f(\xi)e^{ix\xi}\, d\xi. 
$$
By Plancherel Theorem, the operator
$\cF$ extends to a surjective
isometry $\cF: L^2(\bbR)\to L^2(\bbR)$. 
\smallskip

We now recall the  classical Paley-Wiener characterization 
of the space $PW_a$.

\begin{Theorem}[\cite{PW}]
 Let $f\in PW_a$, then $\supp\widehat f_0\subseteq [-a,a]$, 
 \begin{equation*}
  f(z)=\frac{1}{\sqrt{2\pi}}\int_{-a}^a \widehat{f_0}(\xi)e^{iz\xi}\,d\xi
 \end{equation*}
 and $\|f\|_{PW_a}=\|\widehat{f_0}\|_{L^2([-a,a])}$. 
Conversely, if $g\in L^2([-a,a])$ and we define 
\begin{equation*}
f(z)=\frac{1}{\sqrt{2\pi}}\int_{-a}^ag(\xi)e^{iz\xi}\, d\xi \,,
\end{equation*}
then $f\in PW_a$, $\widehat{f_0}=g$ and $\|f\|_{PW_a}=\|g\|_{L^2([-a,a])}$.
\end{Theorem}

In particular, the Fourier transform $\cF$ induces a surjective
isometry between the spaces $L^2([-a,a])$ 
and $PW_a$.  We shall write $L^2_a$ instead of $L^2([-a,a])$ for short.

A similar characterization holds true for the Bernstein spaces
$\cB^{p}_a$, $1< p< +\infty$. We refer the reader, for instance, to
\cite[Theorem 4]{Andersen}.  We shall denote by $\bbN_0$ the set of
nonnegative integers.

\begin{Theorem}[Characterizations of $\cB^p_a$] Let $1<p<\infty$.
Then, the following conditions on a function $h$ defined on the real
line are equivalent.
\begin{itemize}
\item[(i)] The function $h$ is the restriction of an
entire function $f\in \cB^p_a$ to the real line, that is, $h=f_0$;
\item[(ii)] $h\in
 L^p(\bbR)$ and $\supp\widehat h \subseteq [-a,a]$;
\item[(iii)]
$h\in C^\infty$,  $h^{(n)}\in L^p$ for all
$n\in\bbN_0$ and $\|h^{(n)}\|_{L^p} \le a^n \|h\|_{L^p} $.
\end{itemize}
\end{Theorem}

The above theorem holds in the limit cases $p=1$ and $p=+\infty$ as well,
but in this paper we only focus on the range $1<p<+\infty$. 

We remark that in the Paley--Wiener characterization of the Bernstein
spaces, the Fourier transform of $f_0\in L^p(\bbR)$ is to be
understood in the sense of tempered distributions. Namely,
$L^p(\bbR)\subseteq \cS'$, and
the Fourier transform extends to a isomorphism of $\cS'$ onto
itself, where $ \widehat f$ is defined by the formula 
$$
\la\widehat f,\vp\ra=\la f,\widehat\vp\ra, \qquad
f\in\cS', \vp\in\cS \,. 
$$

The Paley--Wiener and Bernstein spaces are classical and deeply studied
for several reasons. A well-studied problem for these spaces, for
instance, is the sampling problem and we refer the reader to 
\cite{Seip}, \cite{OS} and references therein. Moreover, the
Paley--Wiener space $PW_a$ is the most important example of a de
Branges space, which are spaces of entire functions
introduced by L. de Branges in \cite{dB}. They have deep connections
with  canonical systems and have been extensively studied in the recent
years. For an overview of de Branges spaces and canonical systems we
refer the reader, for instance, to \cite{Romanov}.  

In this paper we introduce a family of spaces which generalizes the
classical Paley--Wiener and 
Bernstein spaces; we deal with spaces of entire functions of
exponential type $a$ whose restriction to the real line belongs to some
homogeneous Sobolev space and we call these spaces fractional
Paley--Wiener and Bernstein spaces. 
The investigation of these
spaces is not only motivated from the mere will to extend some
classical results, but from the fact that these spaces arise very
naturally in the several variable setting. In order to recover some
classical 1-dimensional results in higher dimension, such as a
Shannon-type sampling theorem, it is necessary to work with suitable
defined fractional Paley--Wiener spaces on $\bbC^{n+1}$. We refer the
reader to \cite{AMPS,MPS} for details and results in the several
variable setting.
In the present work we
start such investigation: we introduce the spaces, we study some of
their structural properties, we prove a Paley--Wiener type
characterization and generalizations of the classical Bernstein and
Plancherel--P\'olya inequalities.
We also point out that classical results such as
sampling theorems for the Paley--Wiener space do not necessarily
extend to the fractional setting (Section \ref{sampling}). 
Finally we mention the papers \cite{P1,P2} in which  the authors studied
other generalizations of the Paley--Wiener spaces. 
\ms

We now precisely define the function spaces we are interested
in. Given a function $f\in\cS$ and $s>0$, we define its fractional
Laplacian $\Delta^{\frac s2}f$ as 
\begin{equation*}
\Delta^{\frac s2}f:= \cF^{-1}(|\cdot|^s\cF f)
\end{equation*}
and we set
$$
\|f\|_{s,p}:=\|\Delta^{\frac s2}f\|_{L^p}.
$$

We remark that for $f\in\cS$ the fractional Laplacian
$\Delta^{\frac s2}f$ is a well-defined function and that
$\|\cdot\|_{s,p}$ is a norm on the Schwartz space (see, for instance,
\cite{MPS-sob}). Therefore, we define the homogeneous Sobolev space
$\dot{W}^{s,p}$ as the closure of $\cS$ with respect to
$\|\cdot\|_{s,p}$, i.e., 
\begin{equation*}
 \dot{W}^{s,p}=\overline{\cS}^{\|\cdot\|_{s,p}}.
\end{equation*}

As described in \cite{MPS-sob}, the space $\dot{W}^{s,p}$ turns out to
be a quotient space of tempered distributions modulo polynomials of
degree $m=\lfloor s-1/p\rfloor $, where we denote by $\lfloor x\rfloor $ the integer part of
$x\in\bbR$ and by
$\cP_m$
the set
polynomials of degree at most $m$, where $m\in\bbN_0$.
In \cite[Corollary 3.3]{MPS-sob} we prove that $f\in\dot{W}^{s,p}$ if and only if  
\begin{itemize}
 \item[{\tiny $\bullet$}] $f\in\cS'/\cP_m$;
 \item[{\tiny $\bullet$}] there exists a sequence $\{f_n\}\subseteq \cS$
   such that $f_n\to f$ in $\cS'/\cP_m$;  
 \item[{\tiny $\bullet$}] the sequence $\{\Delta^{\frac s2}f_n\}$ is a Cauchy
   sequence with respect to the $L^p$ norm.
 \end{itemize}
 If $f\in\dot{W}^{s,p}$ we then set 
\begin{equation}\label{frac-lapl-2}
\Delta^{\frac s2}f=\lim_{n\to+\infty}\Delta^{\frac s2}f_n,
\end{equation}
where the limit is to be understood as a limit in the $L^p$ norm.

In order to avoid working in a quotient space, instead of considering
the spaces $\dot{W}^{s,p}$, we consider the realization spaces $E^{s,p}$, 
see \cite[Corollary 3.2]{MPS-sob}.    Inspired by the works of
G. Bourdaud
\cite{Bourdaud, bourdaud2, bourdaud3}, if $m\in\bbN_0\cup\{\infty\}$ and
$\dot{X}$ is a given  subspace  of
$\cS'/\cP_m$ which is a Banach space, such that the natural inclusion of
$\dot{X}$ into $\cS'/\cP_m$ is continuous,
we call a subspace $E$ of $\cS'$ a {\em realization} of $\dot{X}$ if there exists a bijective linear map
$$
R: \dot{X}\to E
$$ 
such that $\big[ R[u] \big] = [u]$ for every equivalence class $[u]\in\dot{X}$. 
We endow $E$ of the norm given by $\| R[u]\|_E = \| [u] \|_{\dot{X}}$.

For $\gamma>0$, we denote by $\dot{\Lambda}^{\gamma}$ the homogeneous Lipschitz space
of order $\gamma$.
Given a locally integrable function,  and  an
interval $Q$, we 
denote by $f_Q$ the average of $f$ over $Q$, and denote by $\BMO$ the
standard space of functions (modulo constants) of bounded mean
oscillation.  Finally, for a sufficiently smooth function $f$, we denote
by $P_{f;m;0}$ the Taylor polynomial of $f$ of order $m$ at the origin.
The next result describes the realization spaces $E^{s,p}$.

\begin{Theorem}[\cite{MPS-sob}]\label{real-hom-sob} 
For $s>0$ and $p\in(1,+\infty)$, let $m=\lfloor s-\frac 1p\rfloor$. Then, 
$\dot{W}^{s,p}\subseteq \cS'/\cP_m $.  We define the spaces $E^{s,p}$
as follows.

\noindent
{\rm (i)}
Let $0<s<\frac 1p$, and let $p^*\in(1,\infty)$ given by
$\frac1p - \frac{1}{p^*}=s$, define
 $$
E^{s,p}=\big\{f\in L^{p^*}:\,
\| f\|_{E^{s,p}} := \|\Delta^{s/2} f\|_{L^p} <+\infty \big\} \,.
 $$

\noindent
{\rm (ii)}  Let $s-\frac 1p\in \bbR^{+}\backslash \bbN$  let
 $m=\lfloor s-\frac 1p\rfloor $, and define
$$
 E^{s,p}
=\Big\{ 
f\in\dot{\Lambda}^{s-\frac 1p}:\, 
\ P_{f;m;0}=0\,,\  
\| f\|_{E^{s,p}} := \|\Delta^{s/2} f\|_{L^p}<+\infty\, \Big\}\,.
$$

\noindent
{\rm (iii)} Let $s-\frac 1p\in\bbN_0$. Fix the bounded interval
$Q=[0,2\pi]$.  
If $s=\frac 1p$, let
\begin{equation*} 
E^{s,p}
=\Big\{ 
f\in\BMO:\,  f_Q=0,\,
\| f\|_{E^{s,p}} := \|\Delta^{\frac s2}f\|_{L^p}<+\infty \Big\}\,. 
\end{equation*} 
 If $m=s-\frac1p$ is a positive integer, define
 \begin{equation*} 
E^{s,p}
=\Big\{ 
f\in\cS'\cap\cC^{m-1}: P_{f;m-1;0}=0,\, f^{(m)}\in \BMO,\,
f^{(m)}_Q=0,\, 
\| f\|_{E^{s,p}} :=\|\Delta^{\frac s2}f\|_{L^p}<+\infty \Big\}\,. 
\end{equation*} 

Then, the space $E^{s,p}$ is a realization space for $\dot{W}^{s,p}$.
\end{Theorem}

When restricted to $E^{s,p}$, $\|\cdot\|_{s,p}$ is no longer a
semi-norm, but a genuine norm. In particular, the fractional Laplacian
on $E^{s,p}$ is injective.
\ms

We are now ready to define the fractional Bernstein spaces.

\begin{defn}\label{frac-bernstein-defn} {\rm
 For $a,s>0$ and $1<p<+\infty$ the fractional Bernstein space $\cB^{s,p}_a$ is defined as
 \begin{equation*}
\Bsp=\big\{f\in \cE_{a}:\,  f_0\in E^{s,p} \text{ and, if } s\ge
1/p \text{ and } \, m=\lfloor s-1/p\rfloor,\,  P_{f_0;m;0}=0\big\}\,.
\end{equation*}
We endow the space $\cB^{s,p}_a$ with the norm
$\|f\|_{\cB^{s,p}_a}:=\|f_0\|_{E^{s,p}}$. 
} \ms
\end{defn}

\begin{remark}{\rm
In this paper we restrict ourselves to the case
$s-\frac1p\notin\bbN_0$, $s>0$, $p\in (1,\infty)$. 
The case $s-\frac1p\in\bbN_0$ could be thought to be the {\em
  critical} case, as in the Sobolev embedding theorem.
All the proofs  break down for these values of $s$ and $p$, although
we believe that all the results in this paper extend also to case
$s-\frac1p\in\bbN_0$. 

Thus, the case $s-\frac1p\in\bbN_0$ remains open and is, in our opinion,
of considerable interest.  We will add some comments on this problem
in the final Section \ref{Final-Sect}.
} 
\end{remark}
\smallskip

\begin{remark}{\rm
We point out that from the results in the present work we can easily
deduce analogous results for the {\em homogeneous} fractional
Bernstein spaces $\dot{\cB}^{s,p}_a$, defined as above, but without requiring that
$P_{f_0;m;0}=0$.  In this way, we obtain spaces of entire functions of
exponential type {\em modulo} polynomials of degree $m=\lfloor s-\frac1p\rfloor$.
}
\end{remark}

We first consider the spaces $PW^s_a$, $s>0$, and we prove some
Paley--Wiener type theorems assuming that $s-\frac12\notin\bbN_0$. For any $s>0$ let 
$L^2_a(|\xi|^{2s})$ be the weighted $L^2$-space 
\begin{equation*}
L^2_a(|\xi|^{2s})=\bigg\{f:[-a,a]\to\bbC\, \text{\ such that\ }
  \int_{-a}^a|f(\xi)|^2|\xi|^{2s}\,d\xi<\infty \bigg\}. 
\end{equation*}

We prove the following Paley--Wiener type theorems. We distinguish the
case $0<s<\frac12$ from the case $s>\frac12$.
\begin{theorem}\label{main-1}
Let $0<s<\frac12$ and let $f\in PW^s_a$. Then, 
$\supp\widehat
{f_0}\subseteq [-a,a]$, 
$\widehat
{f_0} \in L^2_a(|\xi|^{2s})$ and 
\begin{equation}\label{PW-1-eq1}
 f(z)=\frac{1}{\sqrt{2\pi}}\int_{-a}^a\widehat f_0(\xi) e^{iz\xi}\, d\xi\,.
\end{equation}
Moreover, 
$\|f\|_{PW^s_a}=\|\widehat{f_0}\|_{L^2_a(|\xi|^{2s})}$. 
 Conversely, let $g\in L^2_a(|\xi|^{2s})$, and define $f$ by setting
 \begin{equation}\label{PW-1-eq2}
  f(z)=\frac{1}{\sqrt{2\pi}}\int_{-a}^{a}g(\xi)e^{iz\xi}\, d\xi.
 \end{equation}
  Then, $f\in PW^s_a$, $\widehat {f_0}=g$ and
 $\|f\|_{PW^s_a}=\|g\|_{L^2_a(|\xi|^{2s})}$. 
\end{theorem}

\begin{defn}\label{Ksm-def}{\em 
    Given $s>\frac12$, let
    $m=\lfloor s-\frac12\rfloor$, for any  $g\in
L^2(|\xi|^{2s})$ we define  $Tg$ by setting, for $\psi\in\cS$,
\begin{equation}\label{defn-Tg}
  \la Tg,\psi\ra:=\frac{1}{\sqrt{2\pi}}\int_\bbR g(\xi)
  \big( \psi(\xi)-P_{\psi;m;0}(\xi) \big)\, d\xi.
\end{equation}
As we will see, Lemma \ref{lemma-Tg}, $Tg$ is well-defined for any
$\psi\in\cS$, in particular $Tg\in\cS'$, and $T: L^2(|\xi|^{2s})\to
\cS'$ is a continuous operator. 
}
\end{defn}

We denote by $\cD'_c$ the space of distributions with compact support,
which is the dual of $C^\infty$. 

\begin{theorem}\label{main-2}
Let $s>\frac12$, $m=\lfloor s-\frac12\rfloor$, assume that $s-\frac12\notin\bbN$  and set
$P_m(iz\xi)=\sum_{j=0}^{m}(iz\xi)^j/j!$.   
 Let $f\in PW^s_a$, then $\supp\widehat{f_0}\subseteq [-a,a]$ and there exists  $g \in
L^2_a(|\xi|^{2s})$ such that $\widehat f_0=Tg$ in $\cD'_c$, 
 and
  \begin{align}\label{PW-poly-eq1}
  \begin{split}
f(z)&=\la\widehat f_0, e^{iz(\cdot)}\ra=\frac{1}{\sqrt{2\pi}}\int_{-a}^a g(\xi) \big(
e^{iz\xi}-P_m(iz\xi) \big)\, d\xi.
 \end{split}
 \end{align}
Moreover, $\|f\|_{PW^s_a}=\| g\|_{L^2_a(|\xi|^{2s})}$. Conversely, let $g\in L^2_a(|\xi|^{2s})$
and define $f$ by setting  
\begin{equation}\label{PW-poly-eq2}
f(z)=\la Tg, e^{iz(\cdot)}\ra=\frac{1}{\sqrt{2\pi}}\int_{-a}^a g(\xi) \big(e^{iz\xi}-P_m(iz\xi) \big)\, d\xi.
\end{equation}
Then, $f\in PW^s_a$ and $\|f\|_{PW^s_a}=\|g\|_{L^2_a(|\xi|^{2s})}$. \ms
\end{theorem}

Observe that in particular Theorem \ref{main-1} 
 says that, if $0<s<\frac12$, the Fourier transform
$\cF: PW_a^s\to L^2_a(|\xi|^{2s})$ is a surjective isomorphism, as in
the case $s=0$.  On the other hand, if $s>\frac12$, $\cF:PW_a^s\to
T(L^2_a(|\xi|^{2s}))$ is a surjective isomorphism, where
$T(L^2_a(|\xi|^{2s}))\subseteq{ \cD'_c}$ denotes 
the image of $L^2_a(|\xi|^{2s})$ via the operator $T$, endowed with
norm $\|Tg\|:=\|g\|_{L^2_a(|\xi|^{2s})}$.  
\ms

As a consequence of the above theorems we obtain that the spaces
$PW^s_a$ are reproducing kernel Hilbert spaces and we are able to make
some interesting remarks concerning reconstruction formulas and
sampling in $PW^s_a$ for $0<s<\frac12$. In particular, we obtain that
the spaces $PW^s_a$ are not de Branges spaces. We 
refer the reader to Section \ref{sampling} below for more details.  

Then we turn our attention to the fractional Bernstein spaces
$\cB_a^{s,p}$.

\begin{theorem}\label{main-3}
  Let $s>0$, $1<p<\infty$ be such that $s-\frac 1p\notin\bbN$.
Then, the fractional Bernstein spaces $\cB^{s,p}_a$ are Banach spaces and the
following Plancherel--P\'olya estimates 
hold. If $0<s<\frac1p$, for $f\in \cB^{s,p}_a$ and $y\in\bbR$ we have
$$
\|f (\cdot+iy)\|_{\cB^{s,p}_a} \le e^{a|y|} \| f\|_{\cB^{s,p}_a}  \,.
$$ 
If $s>\frac1p$ and $s-\frac 1p\notin\bbN_0$,
for $f\in \cB^{s,p}_a$ and $y\in\bbR$ given,  define $F(w)=
f(w+iy) - P_{f(\cdot+iy);m;0}(w)$, $w\in\bbC$.  Then, $F\in \cB^{s,p}_a$ and
$$
\|F \|_{\cB^{s,p}_a} \le e^{a|y|} \| f\|_{\cB^{s,p}_a}  \,.
$$ 
\end{theorem}

\begin{theorem}\label{main-4} 
Let $s>0$ and $1<p<\infty$ such that $s-\frac1p\notin\bbN_0$. Given a function $h$ on the real line, the
following conditions are equivalent. 
\begin{itemize}
\item[(i)] The function $h$ is the restriction of an
entire function $f\in \Bsp$ to the real line, that is, $h=f_0$; 
\item[(ii)] $h \in E^{s,p}$ and $\supp\widehat h \subseteq [-a,a]$;
\item[(iii)]
$h\in C^\infty $ and it is such that  $h^{(n)}\in E^{s,p}$ for all
$n\in\bbN_0$ and $\|h^{(n)}\|_{E^{s,p}} \le a^n \|h\|_{E^{s,p}} $.
\end{itemize}
\end{theorem}

Finally, the spaces $PW^s_a$ are closed subspaces of the Hilbert spaces
$E^{s,2}$, and thus there exists a Hilbert space projection operator
$\bP_s: E^{s,2}\to PW^s_a$. It is natural to study the mapping property of the
operator $\bP_s$ with respect to the $L^p$ norm.
We prove the following result.   

 \begin{theorem}\label{main-5}
 Let $s>0$ and $1<p<\infty$ such that $s-\frac12\notin\bbN_0$,
 $s-\frac1p\notin\bbN_0$ and $\lfloor s-\frac12\rfloor=\lfloor s-\frac
 1p\rfloor$. Then, the Hilbert space projection operator $\bP_s:
 E^{s,2}\to PW^s_a$ 
  densely defined on $E^{s,p}\cap E^{s,2}$ extends to a bounded
  operator $\bP_s:E^{s,p}\to\Bsp$ for all $s>0$ and
  $1<p<+\infty$. 
 \end{theorem}

 \smallskip

The paper is organized as follows. After recalling some preliminary
results in Section \ref{pre}, we prove Theorem \ref{main-1} and
\ref{main-2} in Section \ref{PW-sec}. In
Section \ref{Bsp} we investigate the fractional Bernstein spaces
proving Theorems \ref{main-3} and \ref{main-4}, whereas in Section
\ref{sampling} we shortly discuss the 
sampling problem for the fractional Paley--Wiener spaces. Finally, we prove
prove Theorem \ref{main-5} in Section \ref{bdd-projection}, and
conclude with further remarks and open questions in Section
\ref{Final-Sect}.
\ms

\section{Preliminaries}\label{pre}

In this section we recall some results of harmonic analysis we will
need in the remaining of the paper.  We omit the proofs of the results
and we refer the reader, for instance, to \cite{Stein}. 
We do not recall the results in their full generality, but only in the
version we need them.  

Let $0<s<1$ so that the function $\xi\to |\xi|^{-s}$ is locally
integrable. Then, the Riesz potential operator $\cI_s$ is
defined on $\cS$ as  
\begin{equation}\label{riesz-defn}
 \cI_s f =\cF^{-1}(|\cdot|^{-s}\widehat f\,).
\end{equation}
Observe that if
$f\in\cS$ and $0<s<1$, then
$
f= \cI_s \Delta^{\frac s2}f = \Delta^{\frac s2} \cI_s f$. 

For $p\in(0,\infty)$ we denote by $H^p$,
the  Hardy space on $\bbR$.  Having fixed $\Phi\in\cS$ with
$\int\Phi=1$, then
\begin{equation}\label{Hardy-sp-def}
H^p
=\big\{ f\in\cS':\, f^*(x):=\sup_{t>0} |f*\Phi_t(x)| \in L^p
\big\} \,,
\end{equation}
where
$$
\|f\|_{H^p}= \|f^*\|_{L^p}\,.
$$
We recall that the definition of $H^p$ is independent of the choice of
$\Phi$ and that, when
$p\in(1,\infty)$, $H^p$ coincides with $L^p$, with equivalence of
norms. 

The Riesz potential operator extends to a bounded operator on $H^p$,
$0<p<\infty$, 
according to the following theorem.  Part (ii) is due to Adams, see
\cite{Adams}.

\begin{thm}\label{riesz-potential}
 Let  $0<s<1$, $0<p<\infty$.

{\rm (i)} If $s<\frac 1p$ and 
 $
 \frac{1}{p^*}=\frac 1p- s$,
 then, $\cI_s$ extends to a bounded operator
 $\cI_s:H^p\to H^{p^*}$.

{\rm (ii)} If 
$s=\frac 1p$, then, $\cI_s$ extends to a bounded operator
$\cI_s: L^p\to \BMO$. 
\end{thm}

\begin{defn}\label{Sm-Sinfty}{\rm For $M$ a nonnegative integer, define
$$ 
\cS_M =\bigg\{f\in\cS: \int_{\bbR} x^k f(x)\, dx=0
\textrm{ for  }k\in\bbN_0,\ k\le M\bigg\}
$$
and
\begin{equation*}
 \cS_{\infty} =\bigg\{f\in\cS: \int_{\bbR} x^k
 f(x)\, dx=0 \textrm{ for all }k\in\bbN_0\bigg\} \,.
 \end{equation*}
}
\end{defn}

We recall that, as it is elementary to verify, $\vp\in\cS_M$, if and
only if there exists $\Phi\in\cS$ such that $\vp=\Phi^{(M)}$.  We will
use this fact several times.   We also recall that 
$\cS_\infty$ is dense in $H^p$ for all $p\in(0,\infty)$, see
\cite[Ch.II,5.2]{Stein}.
For these and other
properties of the Hardy spaces see e.g. \cite{Stein} or
\cite{Grafakos}.

Notice that the Riesz potential operator $\cI_s$ is also
well-defined on $\cS_\infty$ for any $s\ge0$, since if
$f\in\cS_\infty$, then $\widehat f$ vanishes of infinite order
at the origin. Moreover, for all $s>0$
\begin{equation}\label{Delta-s-S-infty}
\cI_s , \Delta^{\frac s2}: \cS_\infty\to \cS_\infty
\end{equation}
are both surjective bounded isomorphisms and in fact one the inverse of the
other one; see  e.g.
\cite[Chapter 1]{Grafakos}.
\ms

\section{Fractional Paley--Wiener spaces}\label{PW-sec}

In this section we first prove
Theorems \ref{main-1} and \ref{main-2}, we deduce that the space
$PW^s_a$ is a reproducing kernel Hilbert space for every $s>0, s-\frac12\notin\bbN_0$,  and 
explicitly compute 
its reproducing kernel. We conclude this section by proving that the
classical Paley--Wiener space $PW_a$ and $PW^s_a$ are actually
isometric. 
\ms

\begin{lem}\label{cpt-support}
Let $f\in  PW^s_a$, $s>0$. Then, $\supp\widehat f_0\subseteq [-a,a]$, so that
$\widehat f_0
\in \cD'_c$, and
$\big(\widehat f_0\big)_{|_{\cP_m}}=0$, where  $m=\lfloor s-\frac12\rfloor$.
\end{lem}

\proof
It is clear from the description of the realization spaces $E^{s,2}$
that $f_0\in\cS'$, hence, once we prove that $\supp\widehat
f_0\subseteq [-a,a]$, it immediately follows that $\widehat f_0\in
\cD'_c$. Let 
$\vp\in\cS_M\cap C^\infty_c$, with $M\ge s$. Given $f\in PW^s_a$ we define
$$
f_{\vp}(z):=\int_{\bbR}f(z-t)\vp(t)\, dt
$$
and we claim that $f_\vp\in \cE_a$ and $(f_\vp)_0=f_0* \vp\in
L^2$; where the symbol  $*$ denotes the standard convolution
on the real line.  
The function $f_\vp$ is clearly entire and, for every $\eps>0$, 
\begin{align*}
|f_\vp(z)|&\leq \int_{\bbR}|f(z-t)|\vp(t)|\, dt
\leq
C_\eps e^{(a+\eps)|z|}\int_{\bbR}e^{(a+\eps)|t|}|\vp(t)|\,
dt
\leq
C e^{(a+\eps)|z|}\
\end{align*}
where the last integral converges since $\vp$ is compactly
supported. Hence, $f_\vp\in\cE_a$.  Moreover, since $\vp\in\cS_M$,
then $\vp\ast\eta\in\cS_M$ as well for any $\eta\in\cS$. Therefore,
if $\{\vp_n\}\subseteq\cS$ is such that $\varphi_n\to f_0$ in
$\cS'/\cP_m$ and $\Delta^{\frac s2} \varphi_n \to \Delta^{\frac s2} f_0$
in $L^2$, for any $\eta\in\cS$ we have\footnote{We warn the reader that, we
  shall denote with the same symbol $\la\cdot,\cdot\ra$ different {\em bilinear}
  pairings of duality, such as $\la\cS',\cS\ra$, $\la\cD'_c,C^\infty\ra$, 
$\la L^{p'},L^p\ra$,
  etc. The actual pairing of duality should be clear from the context
  and there should not be any confusion.}  
\begin{align*}
 \la f_0\ast\varphi,\eta\ra&=\la f_0, \vp\ast\eta\ra\\
 &= \lim_{n\to+\infty}\la\vp_n,\vp\ast \eta\ra\\
 &=\lim_{n\to+\infty}\la\Delta^{\frac s2}\vp_n, \cI_s(\vp\ast\eta)\ra \,.
\end{align*}
The last equality follows
using the Parseval identity,
since $\vp\in\cS_M$, hence $\vp\ast\eta\in \cS_M$ as well, so that
$\cI_s(\vp\ast\eta)\in L^2$.  Moreover, if 
$\Psi\in\cS$ is such that $\vp\ast\eta=\Psi^{(M)}$, we have
$
\cI_s (\vp\ast\eta) = \cI_{s-\ell} \big(R^\ell  \Psi^{(M-\ell)}\big)$,
where $R$
denotes the Riesz transform, and $\ell=\lfloor s \rfloor$. Then,
\begin{align*}
  \|\cI_s(\vp\ast\eta)\|_{L^2}
  =\| R^\ell  \cI_{s-\ell}  \Psi^{(M-\ell)} \|_{L^2}
  \leq C\|\Psi^{(M-\ell)}\|_{H^p}<\infty
\end{align*}
where we have used Theorem \ref{riesz-potential} (i) with
$\frac12=\frac1p-(s-\ell)$. Notice
that $p$ could be either greater or smaller than $1$. If $p>1$, then
$\|\Psi^{(M-\ell)}\|_{H^p}\approx\|\Psi^{(M-\ell)}\|_{L^p}<\infty$ since
$\Psi\in\cS$, if $p<1$, the fact that $\|  \Psi^{(M-\ell)}\|_{H^p}$ is
finite if $M$ is sufficiently large is a well-known 
fact, see e.g. \cite{Stein}. Therefore,
\begin{align*}
\la f_0\ast\vp,\eta\ra&=\lim_{n\to+\infty}\la \Delta^{\frac s2}\vp_n, \cI_{s}(\vp\ast\eta)\ra\\
&=\la \Delta^{\frac s2} f_0, \cI_s(\vp\ast\eta)\ra\\
&=\la \Delta^{\frac s2}f_0\ast \cI_s\vp,\eta\ra
\end{align*}
for any $\eta\in\cS$. In particular $f_0\ast\vp=\Delta^{\frac s2}
f_0\ast \cI_s\vp$ in $L^2$ and if  $\Phi\in\cS$ is such that 
$\Phi^{(M)}=\vp$,  
\begin{align}
\| (f_{\vp})_0 \|_{L^2} & =\|f_0\ast\vp\|_{L^2} =\| \Delta^{\frac s2}
 f_0*R^\ell \cI_{s-\ell}  \Phi^{(M-\ell)}\|_{L^2}  \le \| f\|_{PW^s_a}
 \| \cI_{s-\ell} R^\ell \Phi^{(M-\ell)}\|_{H^1} \notag \\
& \le C \| f\|_{PW^s_a} \|  \Phi^{(M-\ell)}\|_{H^p} <\infty
\,, \label{Riesz-potential-estimate} 
\end{align}
for $M$ sufficiently large, 
where we have used Theorem \ref{riesz-potential} (i) with
$1=\frac1p-(s-\ell)$. This shows that $(f_{\vp})_0 \in L^2$, and
therefore, $f_\vp\in PW_a$.
Thus, $\supp (\widehat{f_0\ast\vp})=\supp (\widehat
f_0\cdot\widehat\vp)\subseteq [-a,a]$. 
Since for every $\xi_0\neq0$, $\xi_0\in\bbR$, there exists
$\vp\in\cS_M$ such that $\widehat\vp(\xi_0) \neq 0$, 
 we conclude that
$\supp\widehat f_0\subseteq [-a,a]$ as we wished to show. \ms 

Let now $s>\frac12$, $s-\frac12 \not\in\bbN$, and let $m=\lfloor s-\frac12\rfloor$.  
Fix $\chi\in C^\infty_c$, $\chi\ge0$, $\chi=1$ on $[-a,a]$.  For
$\eps>0$,
we adopt
here and throughout the paper the notation
$$
\chi^\eps(x)= \chi(\eps x),\quad \chi_\eps (x) =
\textstyle{\frac1\eps} \chi(x/\eps) \,.
$$

Given $Q=\sum_{j=0}^m q_j x^j \in\cP_m$, 
we have
\begin{align*}
\la \widehat f_0, Q\ra 
& = \la \widehat f_0, \chi^\eps Q\ra =  \big\la f_0, \widehat \chi_\eps
  * \Big( 
\sum_{j=0}^m (-1)^j q_j D^j \delta_0 \Big)  \big\ra \\
& = \la \sum_{j=0}^m q_j D^j f_0 ,    \widehat \chi_\eps \ra
 = \la \sum_{j=0}^m q_j D^j f_0 ,    \chi^\eps \widehat \chi_\eps  \ra
+  \la f_0 ,    \sum_{j=0}^m (-1)^j q_j D^j  \Big((1-\chi^\eps) \widehat
\chi_\eps \Big) \ra \\
& = I_\eps + I\!I_\eps \,.
\end{align*}
This equality holds for all $0<\eps\le1$ and we observe that, since
$f_0$ is of moderate growth, both $I_\eps$ and  $I\!I_\eps$ are given
by absolutely convergent integrals.  Let $M>0$ be such that 
$|f_0 (x)|\le C(1+|x|)^M$,
for some $C>0$.   We have
\begin{align*}
|I\!I_\eps|
& \le C \sum_{j=0}^m \int_{|x|\ge \frac a\eps} (1+|x|)^M
\frac{1}{\eps^{j+1}} 
D^j (\widehat \chi) (\frac
x\eps)\, dx \le C
\sum_{j=0}^m \int_{|x|\ge \frac a\eps} (1+|x|)^{M+m}
|D^j (\widehat \chi) (\frac
x\eps)|\, dx \\
& =  C\eps \sum_{j=0}^m \int_{|t|\ge \frac{a}{\eps^2}} (1+|\eps t|)^{M+m}
|D^j  \widehat \chi(t)|\, dt \\
& \le C_N \eps^N\,,
\end{align*}
for any $N>0$. On the other hand, using Lebesgue's dominated
convergence theorem it is easy to see that, as $\eps\to 0$,
\begin{align*}
I_\eps
& =  \int Q(D)f_0(\eps t)    \chi(\eps^2 t) \widehat\chi(t)\, dt
\to Q(D)f_0(0) \int \widehat\chi(t)\, dt =0\,,
\end{align*}
since $P_{f_0;m;0}=0$ and $Q\in\cP_m$.  Hence, $\widehat f_0(Q)=0$ and we are
done. 
\qed \ms

We now prove our first main theorem.
\proof[Proof of Theorem \ref{main-1}]
We start proving the second part of the statement. Let $g\in
L^2_a(|\xi|^{2s})$ and define $f$ as in \eqref{PW-1-eq2}. Then,  since
$0<s<\frac12$, for $z=x+iy$,
\begin{align*}
 |f(z)|&=\Big|\frac{1}{\sqrt{2\pi}}\int_{-a}^a g(\xi)e^{iz\xi}\, d\xi\Big|
\leq C\|g\|_{L^2_a(|\xi|^{2s})}\left(\int_{-a}^a|\xi|^{-2s}e^{-2y\xi}\right)^{\frac{1}{2}}
 \leq C e^{a|y|}\|g\|_{L^2_a(|\xi|^{2s})}.
\end{align*}
Therefore, $f$ is well-defined, is clearly entire and belongs to
$\cE_a$.  We wish to show that $f_0\in E^{s,2}$.  
Observing that 
\begin{align*}
f_0(x) 
&= \frac{1}{\sqrt{2\pi}} \int_{-a}^a g(\xi) e^{ix\xi}\, d\xi =
\cI_s \cF^{-1}( g|\xi|^s) (x) \,,
\end{align*}
and since $\cI_s:L^2\to L^{2^*}$, we see that $f_0\in  L^{2^*}$.
Moreover, since  $\widehat f_0 = g\in L_a^2(|\xi|^{2s})$, it follows
that $\Delta^{\frac s2} f_0\in L^2$. Hence  $f_0\in E^{s,2}$, $f\in
PW_a^s$ and
$$
\| f\|_{PW^s_a} 
= \| g\|_{L_a^2(|\xi|^{2s})} \,.
$$

Now, let $f\in PW^s_a$. Lemma \ref{cpt-support} 
guarantees that
$\widehat f_0 \in L^2_a(|\xi|^{2s})$ and in particular
 is compactly supported in $[-a,a]$.  From the first part of the theorem, we know that
 the function 
 $$
 \widetilde f(z):=\frac{1}{\sqrt{2\pi}}\int_{-a}^a \widehat f_0(\xi) e^{iz\xi}\, d\xi
 $$
is a well-defined function in $PW^s_a$ and $\widetilde f_0=
f_0$. Hence,  $f$ and $\widetilde f$ coincide everywhere as we wished 
to show. \qed
\ms

\begin{cor}\label{kernel-s-small}
 The spaces $PW^s_a$, $0<s<\frac12$, are reproducing kernel Hilbert
 spaces with reproducing kernel
 \begin{equation*}
K(w,z)=\frac{1}{2\pi}\int_{-a}^a e^{i(w-\bar z)\xi}|\xi|^{-2s}\, d\xi.  
 \end{equation*}
\end{cor}
\proof
From \eqref{PW-1-eq1} we deduce that point-evaluations are bounded on
$PW^s_a$. In fact,   
\begin{align*}
|f(z)|=\left|\frac{1}{\sqrt{2\pi}}\int_{-a}^a \widehat{f_0}(\xi) e^{iz\xi}\,
  d\xi\right|\leq
C\|\widehat{f_0}\|_{L^2_a(|\xi|^{2s})}\Big(\int_{-a}^a
e^{-2y\xi}|\xi|^{-2s}\, d\xi\Big)^{\frac12}\leq C
e^{a|y|}\|\widehat{f_0}\|_{L^2_a(|\xi|^{2s})}. 
\end{align*}
This easily implies that
$PW^s_a$ is complete, hence a reproducing kernel Hilbert space. For
$z\in\bbC$, the
kernel  function $K_z$ satisfies\footnote{We denote by
  $\la\cdot \,|\,\cdot\ra_H$ the {\em Hermitian} inner product on a given
    Hilbert space $H$.}
\begin{align*} 
\frac{1}{\sqrt{2\pi}}\int_{-a}^a\widehat{f_0}(\xi)
 e^{iz\xi}\, d\xi =  f(z)=\big\la
  f\,|\,K_z\big\ra_{PW^s_a}=\big\la\widehat{f_0} \,|\, \widehat
 {(K_z)_0}\big\ra_{L^2_a(|\xi|^{2s})}=\int_{-a}^a \widehat f_0(\xi)
 \overline{\widehat{(K_z)_0}(\xi)}
 |\xi|^{2s\,}d\xi \,. 
\end{align*}
Therefore,  $\widehat{(K_z)_0}(\xi)=\frac{1}{\sqrt{2\pi}} e^{-i\bar
  z\xi}|\xi|^{-2s}\chi_{[-a,a]}(\xi)$ and the conclusion follows.
\epf
\ms

Next, we consider the case $s>\frac12$.

\begin{lem}\label{lemma-Tg}
Let $s>\frac 12$,  $s-\frac12\notin\bbN$ and let $m=\lfloor
s-\frac12\rfloor$. Given $g\in L^2(|\xi|^{2s})$ define $Tg\in\cS'$ as 
\begin{equation}\label{Tg}
  \la Tg,\psi\ra:=\frac{1}{\sqrt{2\pi}}\int_\bbR g(\xi)
  \big(\psi(\xi)-P_{\psi;m;0}(\xi) \big)\, d\xi 
\end{equation}
for any $\psi\in\cS$. Then, $Tg$ is well-defined and $T:
L^2(|\xi|^{2s})\to \cS'$ is a continuous operator.
\end{lem}
\begin{proof}
By H\"older's inequality we have
\begin{align*}
  |\la Tg,\psi\ra|
  &\leq \|g\|_{L^2(|\xi|^{2s})}\bigg(\int_\bbR |\psi(\xi)-P_{\psi;m;0}(\xi)|^2|\xi|^{-2s}\bigg)^\frac12\\
&\leq  \|g\|_{L^2(|\xi|^{2s})}\bigg(\Big(\int_{|\xi|\leq 1}+\int_{|\xi|\geq 1}\Big) |\psi(\xi)-P_{\psi;m;0}(\xi)|^2|\xi|^{-2s}\, d\xi\bigg)^\frac12
\end{align*}
Notice that, since $m=\lfloor s-\frac12\rfloor$ and $s-\frac12\notin\bbN$, we have $-\frac32<m-s<-\frac12$. Hence, 
\begin{align*}
 \int_{|\xi|\geq 1} |\psi(\xi)-P_{\psi;m;0}(\xi)|^2|\xi|^{-2s}\,
  d\xi&\leq\int_{|\xi|\geq 1}|\psi(\xi)|^2|\xi|^{2s}\, d\xi
        +\int_{|\xi|\geq 1}|P_{\psi;m;0}(\xi)|^2|\xi|^{-2s}\, d\xi \\
 &\leq \int_{|\xi|\geq 1}|\psi(\xi)|^2|\xi|^{2s}\, d\xi+\max_{0\leq j\leq m}|\psi^{(j)}(0)|\Big(\int_{|\xi|\geq 1}|\xi|^{2(m-s)}\Big)^{\frac12}\\
 & \leq C_\psi,
\end{align*}
where $C_\psi$ denotes a finite positive constant bounded by some Schwartz seminorm of $\psi$. Moreover,
\begin{align*}
  \int_{|\xi|\leq 1}|\psi(\xi)-P_{\psi;m;0}(\xi)|^2|\xi|^{-2s}\, d\xi
  &\leq\sup_{|\xi|\leq 1}|\psi^{(m+1)}(\xi)|\int_{|\xi|\leq 1}|\xi|^{2(m+1-s)}\, d\xi<+\infty.
\end{align*}
From these estimates it is clear that $T:L^2(|\xi|^{2s})\to\cS'$ is a continuous operator as we wished to show.
\end{proof}

\begin{lem}\label{f-not-hat-weighted-L2}
Let  $s>\frac 12$, $s-\frac12\notin\bbN$ and
$m=\lfloor s-\frac12\rfloor$.  Given $f\in PW^s_a$  there exists a
unique $g\in
L^2_a(|\xi|^{2s})$ such that $\widehat f_0= Tg$ in $\cS'$, that is, 
$$
\la \widehat f_0,\psi\ra=\la Tg,\psi\ra=
\frac{1}{\sqrt{2\pi}}\int_{-a}^a g(\xi) \big(\psi(\xi)-P_{\psi;m;0}(\xi)\big)\, d\xi
$$
for all $\psi\in \cS$.
\end{lem}
\proof
By the results in \cite{MPS-sob}, since 
$f_0\in  E^{s,2}$, 
there exists a sequence $\{\vp_n\}\subseteq\cS$ such that 
 $\{\Delta^{\frac s2} \vp_n\}$ is a Cauchy sequence in $L^2$, and
$\vp_n\to
f_0$ in $\cS'/\cP_m$, where $m=\lfloor s-\frac12\rfloor$, that is,
$
\la \vp_n,\psi\ra \to \la f_0,\psi\ra = \la \widehat f_0, \widehat \psi\ra$,
as $n\to\infty$, for all $\psi\in\cS_m$.  Therefore,
$$
\la \widehat\vp_n, \eta\ra \to \la \widehat f_0,\eta\ra
$$
as $n\to\infty$, for all $\eta\in\widehat{\cS_m}=\cS\cap
\{\eta\in\cS:\, P_{\eta;m;0}\}=0$.  Moreover, there exists a unique
$g\in L^2(|\xi|^{2s})$ such that $\widehat \vp_n\to g$ in
$L^2(|\xi|^{2s})$.  Since $T: L^2(|\xi|^{2s})\to \cS'$ is continuous,
we also have that $T\widehat\vp_n\to Tg$ in $\cS'$. We now prove that
it holds also $T\widehat \vp_n\to \widehat f_0$ in
$(\widehat{\cS_m})'$ In fact, given $\psi\in \widehat {\cS_m}$, we
have 
\begin{align*}
  \la T\widehat \vp_n,\psi\ra
  &=\int_{\bbR}\widehat\vp_n(\xi) \big( \psi(\xi)-P_{\psi;m;0}(\xi) \big)\, d\xi
 =\int_\bbR \widehat\vp_n(\xi)\psi(\xi)\, d\xi=\int_\bbR \vp_n(\xi)\widehat\psi(\xi)\,d\xi\\
 &\to\int_\bbR f_0(\xi)\widehat\psi(\xi)\, d\xi = \la\widehat
   f_0,\psi\ra \,.
\end{align*} 
Therefore, $\widehat f_0= Tg$ in $(\widehat{\cS_m})'$, that is, if $Q(D)(\delta)=\sum_{j=0}^{m}c_j \delta^{(j)}$,
$$
\widehat f_0= Tg+Q(D)(\delta) 
$$
in $\cS'$. In particular, this implies that $\supp Tg\subset[-a,a]$, hence, $Tg\in\cD'_c$ and $\supp g\subseteq[-a,a]$.

We now prove that $Q(D)(\delta)=0$. Let $P\in \cP_m$ and let $\eta\in
C^\infty_c$ such that $\eta\equiv 1$ on $[-a,a]$ so that $\eta P\in
\cS$. Since $\widehat f_0$ is supported in $[-a,a]$ from Lemma
\ref{cpt-support} we get 
\begin{align*}
 \la \widehat f_0, \eta P\ra=\la\widehat f_0, P\ra=0
\end{align*}
and, since $Tg$ is supported in $[-a,a]$ as well,
\begin{align*}
 \la Tg, \eta P\ra&=\int_{-a}^a g(\xi) \big((\eta P)(\xi)- P_{\eta P,m;0}(\xi)\big)\, d\xi=0
\end{align*}
since $P_{\eta P,m;0}=\eta P$ on $[-a,a]$. Therefore, we obtain that
$\la Q(D)(\delta), \eta P\ra=0$ as well and, by the arbitrariness of
$\eta P$, we conclude that $Q(D)(\delta)=0$ as we wished to
show. Thus, $\widehat f_0= Tg$ in $\cS'$. 
\qed\ms

Before proving the next lemma, we need the following definition. Given
$s>0$ and $\psi\in\widehat{\cS_\infty}$, notice that
$|\xi|^s\psi\in\widehat{\cS_\infty}$. 
Then, given
$U$ in $\cS'$, for any  we define $|\xi|^s U$ by setting
$$
\la |\xi|^sU,\psi\ra= \la U, |\xi|^s\psi\ra \,.
$$
 We now prove the following simple, but
not obvious, lemma.

\begin{lem}\label{Lemma3.3}
 Let $s>\frac12, s-\frac12\notin\bbN$ and let  $f\in PW^s_a$. Then,
 $\cF(\Delta^{\frac s2}f_0)=|\xi|^s\widehat f_0$, with equality in
 $L^2_a$. 
\end{lem}
\begin{proof}
Since $f\in PW^s_a$, we already know that $\cF(\Delta^{\frac
  s2}f_0)\in L^2$. We now consider $|\xi|^s\widehat f_0\in
(\widehat{\cS_\infty})'$ and we show that it actually belongs to
$L^2_a$. Then, we show it coincides with $\cF(\Delta^{\frac s2}f_0)$. 
 Let $\psi\in \widehat{\cS_\infty}$. Then, from Lemma
 \ref{f-not-hat-weighted-L2} there exists $g\in L^2_a(|\xi|^{2s})$
 such that
 \begin{align*}
   \la |\xi|^s\widehat f_0, \psi\ra
   & =\la\widehat f_0, |\xi|^s\psi\ra=\la Tg,|\xi|^s\psi\ra
     =\frac{1}{\sqrt{2\pi}}\int_{-a}^a
   g(\xi) \big(|\xi|^s \psi(\xi)-P_{|\xi|^s\psi;m;0}(\xi) \big)\,
     d\xi \\
   & = \frac{1}{\sqrt{2\pi}} \int_{-a}^a
     g(\xi) |\xi|^s\psi(\xi)\, d\xi \,,
 \end{align*}
 since $P_{|\xi|^s\psi;m;0} =0$.  Hence,
 \begin{align*}
   |\la |\xi|^s\widehat f_0,\psi\ra|
   &\leq C\|g\|_{L^2_a(|\xi|^{2s})}\|\psi\|_{L^2} \,.
 \end{align*}
By the density of $\widehat{\cS_\infty}$ in $L^2$, we conclude that
$|\xi|^s\widehat f_0\in (L^2)'$, that is, $|\xi|^s\widehat f_0\in
L^2_a$ as we wished to show. Now,
since $f_0 \in E^{s,2}$, there exists $\{\vp_n\}\subseteq \cS$ such that $\vp_n\to
f_0$ in $\cS'/\cP_m$ and $\{\Delta^{\frac s2} \vp_n\}$ is a Cauchy
sequence in $L^2$. Then, for $\psi\in\cS_\infty$, which is
dense in $L^2$, we have
\begin{align*}
  \la\Delta^{\frac s2}f,\psi\ra
  &=\lim_{n\to+\infty}\la\Delta^{\frac s2}\vp_n,\psi\ra
    =\lim_{n\to+\infty}\la|\xi|^s\widehat{\vp_n},\widehat\psi\ra
    =\lim_{n\to+\infty}\la\widehat \vp_n,|\xi|^s\widehat
    \psi\ra  \\
  & =  \lim_{n\to+\infty}\la  \vp_n, \cF^{-1} \big( |\xi|^s\widehat
    \psi \big) \ra  = \la f_0, \cF^{-1} \big( |\xi|^s\widehat
    \psi \big) \ra 
    =\la\widehat f_0,|\xi|^s\widehat\psi\ra \\
  & =\la |\xi|^s\widehat f_0,\psi\ra.
\end{align*}
The conclusion follows from the density of $\cS_\infty\subseteq L^2$.
\end{proof}

\proof[Proof of Theorem \ref{main-2}]
We first prove the second part of the theorem. Recall that
$m=\lfloor s-\frac12\rfloor$ is the integer part of $s-\frac12$.  Given $f$ defined
as in \eqref{PW-poly-eq2}, we see that for every $\eps>0$
\begin{align*}
 \left|\frac{1}{\sqrt{2\pi}}\int_{-a}^a
   g(\xi) \big( e^{iz\xi}-P_m(iz\xi) \big)\,d\xi\right|
&\leq
 \left(\frac{1}{\sqrt{2\pi}}\int_{-a}^a|\xi|^{2s}|g(\xi)|^2\,
   d\xi\right)^{\frac12}\left(\int_{-a}^a
   |\xi|^{-2s}|e^{iz\xi}-P_m(iz\xi)|^2\,d\xi\right)^{\frac12}\\ 
 &\leq
 Ce^{a|z|}|z|^{m+1}
 \|g\|_{L^2_a(|\xi|^{2s})}\left(\int_{-a}^a|\xi|^{2(m-s+1)}
\,d\xi\right)^{\frac12}\\  
 & \le C_\eps e^{(a+\eps)|z|}
\end{align*}
since $2(m-s+1)>-1$ and where we have used the inequality
$ \sum_{j=0}^{+\infty} r^j/(j+m+1)! \le e^r$, for $r>0$. 
Hence, $f$ is
well-defined, clearly entire and it belongs to $\cE_a$. 
Since it is clear that
$P_{f;m;0}=0$, it remains to show that $f_0\in E^{s,2}$. We have
\begin{align*}
  f_0^{(m)}(x+h)-f_0^{(m)}(x)
  &=\frac{1}{\sqrt{2\pi}}\int_{-a}^a g(\xi)(i\xi)^m (e^{i(x+h)\xi}-e^{ix\xi})\, d\xi\\
&=\frac{1}{\sqrt{2\pi}}\int_{-a}^ag(\xi)(i\xi)^m e^{ix\xi}(e^{ih\xi}-1)\, d\xi, 
\end{align*}
so that, 
\begin{align*}
  |f_0^{(m)}(x+h)-f_0^{(m)}(x)|
  &\leq \frac{1}{\sqrt{2\pi}}\int_{-a}^a |\xi|^{s}|g(\xi)|e^{ih\xi}-1||\xi|^{m-s}\, d\xi\\
 &\leq C \|g\|^2_{L^2_a(|\xi|^{2s})}\bigg(\int_{-a}^a |e^{ih\xi}-1|^2|\xi|^{2(m-s)}\, d\xi\bigg)^\frac12\\
 &\leq C |h|^{s-\frac12-m}\|g\|^2_{L^2_a(|\xi|^{2s})}\bigg(\int_{\mathbb R} |e^{it}-1|^2|t|^{2(m-s)}\, dt\bigg)^\frac12\\
 &\leq C |h|^{s-\frac12-m}\|g\|^2_{L^2_a(|\xi|^{2s})}\,.
\end{align*}

Hence,
\begin{align*}
\sup_{h\in\mathbb R,h\neq
  0}\frac{|f_0^{(m)}(x+h)-f_0^{(m)}(x)|}{|h|^{s-\frac12-m}}
&\leq C\|g\|^2_{L^2_a(|\xi|^{2s})}
\end{align*}
and we conclude that $f_0\in\dot{\Lambda}^{s-\frac12}$ as we wished to
show (see \cite[Proposition 1.4.5]{Grafakos}).  
Next, we need to show that $\Delta^{\frac s2} f_0\in L^2$ and
$\|\Delta^{\frac s2}f_0\|_{L^2}=\|g\|_{L^2_a(|\xi|^{2s})}$.

To this end, let
 $\{\psi_n\} \subseteq\cS$ be such
 that $\psi_n\to g$ in $L^2_a(|\xi|^{2s})$ and define $\vp_n$
as in  \eqref{PW-poly-eq2}, that is,
$$
\vp_n(x) = \frac{1}{\sqrt{2\pi}}\int_{-a}^a \psi_n(\xi)
\big(e^{ix\xi}-P_m(ix\xi) \big)\, d\xi \,,
$$
where, we recall, $m=\lfloor s-\frac12\rfloor$. 
Observe that, by \eqref{Tg}, $\vp_n(x) = \la T\psi_n,
e^{ix(\cdot)}\ra$.  Given $\eta\in\cS$, using Lemma \ref{lemma-Tg}, we
have that all integrals in 
the equalities that follow converge absolutely and we have that
\begin{align*}
   \lim_{n\to\infty} \la \vp_n, \eta\ra
    & = \lim_{n\to\infty}
     \frac{1}{\sqrt{2\pi}}\int_{-a}^a \psi_n(\xi) \big(\widehat
    \eta(\xi) -P_{m;\widehat\eta;0}(\xi)  \big) \, d\xi dx\\
    & = \lim_{n\to\infty}  \la T\psi_n, \widehat\eta \ra  =  \la Tg, \widehat\eta \ra = \la f_0 , \eta\ra \,.
  \end{align*}
Therefore, 
$\vp_n\to f_0 $ in $\cS'$.   Moreover, 
we have that $D^{m+1} \vp_n = \cF^{-1} \big( (i\xi)^{m+1} \psi_n\big)$
and
setting $s':=s-(m+1)\in
  (-\frac12,\frac12)$ we 
 have that, on Schwartz functions,
$\Delta^{\frac s2}  = R^{m+1} \Delta^{\frac{s'}{2}} D^{m+1}$.
Therefore,
$$
\| \Delta^{\frac s2} \vp_n\|_{L^2} = 
\| \Delta^{\frac{s'}{2}} D^{m+1} \vp_n\|_{L^2} = \| \cF\big(  D^{m+1}
\vp_n\big) \|_{L^2(|\xi|^{2s'})} 
= \| \psi_n\|_{L^2(|\xi|^{2s})}  \,.
 $$
It follows that $\{ \Delta^{\frac s2} \vp_n\} $ is a Cauchy sequence in $L^2$
and that
$$
\|\Delta^{\frac  s2}f_0\|_{L^2}
= \lim_{n\to\infty}  \| \psi_n\|_{L^2(|\xi|^{2s})} =
\|g\|_{L^2_a(|\xi|^{2s})} \,. 
$$
\ms

Let us consider now $f\in PW^s_a$. From Lemmas \ref{cpt-support}  and
\ref{f-not-hat-weighted-L2} we
know 
that $\supp\widehat f_0\subseteq [-a,a]$, and that there exists a
unique  $g\in L^2_a(|\xi|^{2s})$ such that 
$\widehat f_0 =Tg$ in $\cS'$. 
Hence, the function
$$
\widetilde f(z):=\frac{1}{\sqrt{2\pi}}\int_{-a}^a  g (\xi)
\big( e^{iz\xi}-P_m(iz\xi) \big)\,d\xi
$$
is a well-defined function in $PW^s_a$ by the first part of the proof. Moreover, 
$$
D^{m+1}_z\widetilde f_0(x)=\frac{1}{\sqrt{2\pi}}\int_{-a}^a (i\xi)^{m+1}g(\xi) e^{ix\xi}\, d\xi,
$$
so that $\cF(D^{m+1}\tilde{f_0})= (i\xi)^{m+1} g$. On the other hand, we also have that
$$
\cF( f^{(m+1)}_0)=(i\xi)^{m+1}\widehat f_0=(i\xi)^{m+1}Tg.
$$
Now, for $\psi\in\cS$,
\begin{align*}
 \la (i\xi)^{m+1}Tg,\psi\ra&= \la Tg,(i\xi)^{m+1}\psi\ra\\
 &=\frac{1}{\sqrt{2\pi}}\int_{-a}^a g(\xi) \big((i\xi)^{m+1}\psi(\xi)-P_{(i\xi)^{m+1}\psi;m;0}(\xi)\big)\, d\xi\\
 &=\frac{1}{\sqrt{2\pi}}\int_{-a}^a g(\xi) (i\xi)^{m+1}\psi(\xi)\, d\xi\\
 &= \la \cF(D^{m+1}\tilde{f_0}),\psi\ra,
\end{align*}
hence $\cF(f^{(m+1)}_0)=\cF(D^{m+1}\tilde{f_0})$, i.e.,
$f_0^{(m+1)}=\widetilde f^{(m+1)}_0$. Therefore, $f$ and $\widetilde f$
coincide up to a polynomial of degree at most $m$. Since
$P_{f;m;0}=P_{\widetilde f;m;0}\equiv 0$, we get $\widetilde f=f$; in
particular, 
$$
f(z)=\frac{1}{\sqrt{2\pi}}\int_{-a}^a g(\xi) \big(e^{iz\xi}- P_{m}(iz\xi)\big)\, d\xi
$$
and $\|\Delta^{\frac s2}f_0\|_{L^2}=\|g\|_{L^2_a(|\xi|^{2s})}$, as we
wished to show.
\epf\ms

As in the case $s<\frac12$, we have the following
\begin{cor}\label{Pw-kernel}
For $s>\frac12, s-\frac12\notin\bbN$, the spaces $PW^s_a$,  are reproducing kernel Hilbert
 spaces with reproducing kernel 
 \begin{equation*} 
  K(w,z)=\frac{1}{2\pi} \int_{-a}^a \big(e^{iw\xi}-P_m(iw\xi)\big)
\big(e^{-i\bar z\xi}-P_m(-i\bar z\xi)\big)|\xi|^{-2s}\, d\xi \,.
 \end{equation*}
\end{cor}

Notice that, since $K_z=K(\cdot,z) \in PW^s_a$, $P_{K_z;m;0}=0$, that
is, $K_z$ vanishes of order $m$ at the origin, where $m=\lfloor
s-\frac12\rfloor$. 

\proof
From the previous theorem we know that the Fourier transform is a
surjective isometry from $PW^s_a$ onto $T(L^2_a(|\xi|^{2s}))$, the
closed subspace of $\cS'$ endowed with norm
$\|Tg\|:=\|g\|_{L^2_a(|\xi|^{2s})}$. Therefore, $PW^s_a$ are Hilbert
spaces.

Similarly to the proof of Corollary \ref{kernel-s-small} we deduce
from the representation formula \eqref{PW-poly-eq1} that the spaces
$PW^s_a$ are reproducing kernel  Hilbert spaces. Then, 
\begin{align*}
  \frac{1}{\sqrt{2\pi}}\int_{-a}^a \widehat{f_0}(\xi)
 \big(e^{iz\xi}-P_m(iz\xi)\big)\, d\xi =
 f(z)&=\left<f \,|\, K_z\right>_{PW^s_a}= \int_{-a}^a
 |\xi|^{2s}\widehat{f_0}(\xi) \overline{\widehat{(K_z)_0}(\xi)} \,
 d\xi
 \end{align*}
and therefore,
$$
\widehat{(K_z)_0}(\xi)= \frac{1}{\sqrt{2\pi}} \big(
e^{-i\bar z\xi}-P_m(-i\bar z\xi) \big)|\xi|^{-2s} \,.
$$
 From this identity and \eqref{PW-poly-eq1}, the conclusion follows.
\epf

The following lemma is obvious and we leave the details to the reader
(or see the proof of Lemma \ref{density-bernstein}).
\begin{lem}\label{density-PW}
 The space $ \{f\in PW_a : f_0\in\mathcal S_\infty\}$  is dense in $PW_a$.
\end{lem}
\ms

We now show that the fractional Laplacian $\Delta^{\frac s2}$ induces
a surjective isometry from $PW^s_a$ onto $PW_a$.

\begin{thm}\label{hilbert-isometry}
Let $s>0$ and assume $s-\frac12\notin\bbN_0$. Then, the operator
$\Delta^{\frac s2} : PW^s_a\to PW_a$
is a surjective
isometry, whose inverse is $\cI_s$ if $0<s<\frac12$, whereas if
$s>\frac12$ the inverse is 
given by
 \begin{equation}\label{Delta-s-mezzi-inverse}
\big( \Delta^{\frac s2}\big)^{-1} h(z) 
= \frac{1}{\sqrt{2\pi}} \int_{-a}^a \widehat h_0(\xi) |\xi|^{-s} \big(
e^{iz\xi} - P_m(iz\xi)\big)\, d\xi \,,
 \end{equation}
with $h\in PW_a$.
\end{thm}

\proof
We only need to prove the theorem in the case $s>\frac12,
s-\frac12\notin\bbN$.  We recall that from Lemma \ref{Lemma3.3}, if
$f\in PW^s_a$, $\cF(\Delta^{\frac s2} f_0) = |\xi|^s \widehat f_0 \in L^2_a$. Hence, 
the map $f\mapsto \Delta^{\frac s2}f_0$  is clearly an isometry, and
$\supp  \big( \cF(\Delta^{\frac s2} f_0) \big) \subseteq [-a,a]$.
By the classical Paley--Wiener theorem, $\Delta^{\frac s2} f_0$
extends to a function in $PW_a$, that we denote by $\Delta^{\frac s2}
f$.

 Let us focus on
 the surjectivity. Let $h\in PW_a$, then, by the previous lemma, there
 exists a sequence $\{\vp_n\} \subseteq\{h\in PW_a :
 h_0\in \cS_\infty\}$ such that $\|h-\vp_n\|_{PW_a}\to 0$
 as $n\to\infty$.   Set
 \begin{equation}\label{lapl-inverse}
 \Phi_n(z) =
\frac{1}{\sqrt{2\pi}}\int_{-a}^a \widehat{
   \vp_n}(\xi)|\xi|^{-s} \big( e^{iz\xi}-P_m(iz\xi) \big) \, d\xi \,.
 \end{equation}
We observe that, 
since $\vp_n\in\cS_\infty$,  we can write
$$ 
\Phi_n(z) = \frac{1}{\sqrt{2\pi}}\int_{-a}^a \widehat{
   \vp_n}(\xi)|\xi|^{-s}  e^{iz\xi} \, d\xi 
-\frac{1}{\sqrt{2\pi}}\int_{-a}^a \widehat{
   \vp_n}(\xi)|\xi|^{-s}P_m(iz\xi)  \, d\xi 
$$
since both 
the integrals converge absolutely.   
We are going to show that $\Phi_n\in PW^s_a$, $\{\Phi_n\}$ is a Cauchy sequence in
 $PW^s_a$,
and that $\Delta^{\frac s2}
\Phi_n\to g$ in $L^2$.  From these facts the surjectivity follows at once.
As in the proofs of Theorems \ref{main-1} and
\ref{main-2}, we see that $\Phi_n\in PW^s_a$.  Moreover, using
\cite[Corollary 3.4]{MPS-sob} and the fact that $\vp_n\in\cS_\infty$, we see that
$$
\Delta^{\frac s2} \Phi_n = \Delta^{\frac s2} \big( \cF^{-1} (\widehat
\vp_n|\xi|^{-s})\big)  = \vp_n\,.
$$
Hence, $\Delta^{\frac s2} \Phi_n \to g$ in $L^2$, and the surjectivity
follows. 

In order to show that the inverse of $\Delta^{\frac s2}$ has the
expression \eqref{Delta-s-mezzi-inverse}, we observe that $\widehat
h_0(\xi) |\xi|^{-s}  \in L^2_a(|\xi|^{2s})$, so that arguing as in the
proof of Theorem \ref{main-2}, we see that 
$F=\big( \Delta^{\frac s2}\big)^{-1} h \in PW^s_a$.
Now, if $\{g_n\}\subseteq C^\infty_c \big(\{ \delta_n\le|\xi|\le
a-\delta_n\}\big)$ are such that $\delta_n\to0$ and $g_n\to 
\widehat
h_0(\xi) |\xi|^{-s}$ in $L^2_a(|\xi|^{2s})$, using
\cite[Corollary 3.4]{MPS-sob}  again we have
\begin{align*}
\cF\big( \Delta^{\frac s2} F_0\big) (t)  & = \lim_{n\to\infty} 
\cF\Big( \Delta^{\frac s2} 
\frac{1}{\sqrt{2\pi}} \int_{-a}^a g_n(\xi)  \big(
                                           e^{ix\xi} - P_m(ix\xi)\big)\, d\xi \Big) (t)\\
  & = 
\lim_{n\to\infty} \cF\Big(  \Delta^{\frac s2} 
\frac{1}{\sqrt{2\pi}} \int_{-a}^a g_n(\xi)  
e^{ix\xi} \, d\xi \Big) (t)\\
& = \lim_{n\to\infty}  |t|^s g_n(t) = \widehat h_0 (t)
\,. 
\end{align*}
Thus, $\Delta^{\frac s2} F=h$ and the surjectivity follows. \qed
\ms

\section{Fractional Bernstein spaces}\label{Bsp}

In this section we study the fractional Bernstein spaces
and we first show
 that the spaces $\cB^{s,p}_a$ are isometric to the
classical Bernstein spaces $\cB^{p}_a$. The proof is similar to the
Hilbert case, but we have to overcome the fact that Plancherel and
Parseval's formulas are no longer available. 
\smallskip

We need the following density lemma.
\begin{lem}\label{density-bernstein}
Let $1<p<\infty$.   Then, the space $\cT=\{f\in \cE_a:
f_0\in\cS_\infty\}$ is dense in $\cB^p_a$. 
\end{lem}
The proof of such lemma is somewhat elementary but not immediate and
it is postponed to Section \ref{density-sect}.

\begin{thm}\label{banach-isometry}
Let $s>0$ such that $s-\frac1p\notin\bbN_0$. Then, the operator $\Delta^{\frac s2}$ is  a surjective isometry 
$$
\Delta^{\frac s2} :\Bsp\to \cB^{p}_a
$$
and the inverse is as in \eqref{Delta-s-mezzi-inverse} (with $h\in
\cB^{p}_a$). 
\end{thm}
\proof
We first notice that $\Delta^{\frac s2}$ is injective on $\cB^{s,p}_a$
since these spaces  are defined using the
realizations 
$E^{s,p}$ of the homogeneous Sobolev spaces $\dot{W}^{s,p}$. 

We now prove that $\Delta^{\frac s2}f\in \cB^{p}_a$ whenever
$f\in\Bsp$. Due to the characterization of the Bernstein
spaces, $\Delta^{\frac s2}f$ is in $\cB^p_a$ if and only if $\Delta^{\frac
  s2}f_0\in L^p$ and $\supp\widehat{\Delta^{\frac
    s2}f_0}\subseteq[-a,a]$. Let $\vp\in C^\infty_c\cap \cS_M$, with $M$ to be chosen
later. 
Given $f\in\Bsp$ we set  
$$
f_{\vp}(z):=\int_{\bbR}f(z-t)\vp(t)\, dt
$$
and we claim that $f_\vp\in \cE_a$ and $(f_\vp)_0\in L^p(\bbR)$. In
fact, $f_\vp$ is clearly entire and, for every $\eps>0$,  
\begin{align*}
|f_\vp(z)|
&\leq \int_{\bbR}|f(z-t)\vp(t)|\, dt
\leq C_\eps e^{(a+\eps)|z|}
\int_{\bbR}e^{(a+\eps)|t|}|\vp(t)|\, dt<+\infty,
\end{align*}
where the last integral converges since $\vp$ is continuous and
compactly supported. Hence, $f_\vp$ is of exponential type $a$. 
Now, let $(f_\vp)_0=f_0\ast\vp$ be the restriction of
$f_\vp$ to the real line. 
Since $f_0\in E^{s,p}$ there exists a sequence $\{\vp_n\}\subseteq
\cS$ such that $\vp_n\to f_0$ 
 in $\cS'/\cP_m$, where $m=\lfloor s-1/p\rfloor$ and 
$\{\Delta^{\frac s2}\vp_n\}$ is a Cauchy sequence in $L^p$. 
We now argue as in \eqref{Riesz-potential-estimate}.
Since $\vp\in\cS_M$, let $\Phi\in\cS$ be such that $\vp=\Phi^{(M)}$,
so that $\cI_s \vp = R^\ell  \cI_{s-\ell}  \Phi^{(M-\ell)}$, where $R$
denotes the Riesz transform.
 Then we have
\begin{align*}
 \int_{\bbR}|(\vp_n\ast\vp)(x)|^p\, dx
&=\int_{\bbR}|\Delta^{\frac
   s2}\vp_n\ast \cI_s\vp(x)|^p\, dx\\ 
 &\leq \|\Delta^{\frac s2}\vp_n\|_{L^p}\|(R^\ell \cI_{s-\ell}  \Phi^{(M-\ell)}\|_{H^1} \\
& \leq  C
\|\Delta^{\frac s2}\vp_n\|_{L^p}\|
 \Phi^{(M-\ell)}\|_{H^q}  \\
 &\le  C_\vp \|\Delta^{\frac s2}\vp_n\|_{L^p},
\end{align*}
where $1=\frac1q-(s-\ell)$, choosing $\ell=\lfloor s\rfloor$ in such a way that $q\le1$.
 In particular, we get that
$\{\vp_n\ast\vp\}$ is a Cauchy sequence in $L^p$.  
Since $\vp_n*\vp\to f_0*\vp$ in $\cS'/\cP_m$ and $\cS_m$ is dense in
$L^{p'}$, we see that 
 $\vp_n\ast\vp \to f_0\ast\vp=(f_\vp)_0$ in $L^p$.

Therefore, $f_\vp$ is an exponential function of type $a$ whose
restriction to the real line is $L^p$-integrable. Hence,
$f_\vp\in\cB^p_a$ and $\supp\widehat{(f_\vp)_0}=\supp(\widehat
{f_0} \widehat\vp)\subseteq [-a,a]$. From the arbitrariness of $\vp$
we conclude that $\supp\widehat f_0\subseteq [-a,a]$. 

We now argue as in the proof of  Lemma \ref{Lemma3.3}, to show
that also $\supp\widehat {\Delta^{\frac
    s2}f_0} \subseteq [-a,a]$.  Let $\Phi_n \in C^\infty_c
(\{\delta_n\le|\xi|\le a\})$, $\Phi_n\to  \widehat f_0$ in
$\cS'$, and setting $\eta_n 
= \cF^{-1}\big(\Phi_n\big)$ we have $\eta_n\in\cS_\infty$.
Now,  for
$\psi\in\cS$,  
as $n\to\infty$ we have
$$
\la \eta_n,\psi\ra = \la \Phi_n, \widehat\psi\ra
\to \la \widehat f_0 , \widehat\psi\ra = \la f_0,\psi\ra \,,
$$
where the pairings are in $\cS'$.  Thus, $\eta_n\to f_0$  in $\cS'$,
which implies that $\Delta^{\frac s2}\eta_n \to \Delta^{\frac
    s2}f_0$ in $\cS'/\cP$, 
so that
$$
\supp\widehat {\Delta^{\frac
    s2}f_0}  \subseteq \bigcup_n \supp\widehat {\Delta^{\frac
    s2}\eta_n} \cup \{0\}\subseteq 
[-a,a] 
\,,
$$
as we wished to show.

Since $\Delta^{\frac s2}f_0\in L^p$ by
hypothesis, we conclude that $\Delta^{\frac s2}f_0$ is the restriction
to the real line of a function in
$\cB^p_a$, function that we denote by
 $\Delta^{\frac s2}f$. Moreover, we trivially have the equality 
$\|f\|_{\cB^{s,p}_a}=\|\Delta^{\frac s2}f\|_{\cB^p_a}$.

It remains to prove that $\Delta^{\frac s2}$ is surjective. Let
$h\in\cB^{p}_a$. Then, by Lemma \ref{density-bernstein}, there
exists a sequence $\{h_n\} \subseteq \cT= \{h\in\cB^p_a:
h_0\in\cS_\infty\}$ such that $h_n\to h$ in $\cB^p_a$. 

Let $0<s<\frac1p$ and set
$$
F_n(z) =\frac{1}{\sqrt{2\pi}}\int_{-a}^a \widehat {(h_n)_0}(\xi) |\xi|^{-s}e^{iz\xi}\, d\xi.
$$
Then, $F_n\in\cE_a$,  $\Delta^{\frac s2}(F_n)_0= (h_n)_0$
and $\{F_n\}$ is a Cauchy sequence in $\Bsp$ since
$\|F_n\|_{\cB^{s,p}_a}=\|h_n\|_{\cB^p_a}$. In particular, this means
 that $\{(F_n)_0\}$ is a Cauchy sequence in $E^{s,p}$.
 Hence, there exists a limit function $\widetilde F\in E^{s,p}$. We need to
 prove that $\widetilde F$ is the restriction to the real line of a
 function in $\cB^{s,p}_a$.  

Since $s<\frac1p$, by Parseval's identity, we have
$$
F_n(z)=\int_{\bbR} h_n(x)
\cF^{-1} \big( |\xi|^{-s}e^{iz\xi}\chi_{[-a,a]}\big)(x)\, dx \,,
$$
so that
$$
|F_n(z)|\leq \|(h_n)_0\|_{L^p}\|\cF^{-1} \big( |\xi|^{-s}e^{iz\xi}\chi_{[-a,a]}\big)\|_{L^{p'}},
$$
where $p,p'$ are conjugate indices. Observing that 
$$
\cF^{-1}\big(|\xi|^{-s}e^{iz\xi}\chi_{[-a,a]}\big)=
\cI_s\big(\cF^{-1}(e^{iz\xi}\chi_{[-a,a]}) \big)
\,,
$$
from Theorem \ref{riesz-potential} we obtain
$$
\|\cF^{-1}\big(|\xi|^{-s}e^{iz\xi}\chi_{[-a,a]}\big)\|_{L^{p'}}\leq C
\|\cF^{-1} \big(e^{iz\xi}\chi_{[-a,a]} \big)\|_{L^{\frac {p'}{1+sp'}}}\,,
$$
since $\frac {p'}{1+sp'}>1$.
However, for $z$ fixed, 
$$
\cF^{-1}\big(e^{iz\xi}\chi_{[-a,a]} \big)(t)=\frac{a}{2\pi}\sinc (a(z+t))
$$
belongs to $ \cB^q_a$ for any $q\in(1,\infty)$. Therefore, by
the classical Plancherel--P\'olya Inequality, we obtain
$$
\|\cF^{-1}\big(|\xi|^{-s}e^{iz\xi}\chi_{[-a,a]}\big)\|_{L^{p'}}\leq C
e^{a|y|}\|\sinc(a\xi)\|_{L^{\frac{p'}{1+sp'}}} \,,
$$
where $y=\Im z$. 
In conclusion,
$$
|F_n(z)|\leq C e^{a|y|}\|(h_n)_0\|_{L^p}.
$$
Since $\|(h_n)_0\|_{L^p}=\|h_n\|_{\cB^p_a}$, we just proved that
the $\cB_a^p$-convergence of $\{h_n\}$ implies
the uniform convergence on compact subsets of $\bbC$ of
$\{F_n\}$ to a function $F$ of exponential type $a$. Necessarily,
$F|_{\bbR}=\widetilde F$   as we wished to
show.   Notice that we also have 
that
$$
\Delta^{\frac s2}F_0=\lim_{n\to+\infty}
\Delta^\frac s2 (F_n)_0=\lim_{n\to+\infty} (h_n)_0=h_0 \,,
$$
that is, the
inverse is given by equation \eqref{Delta-s-mezzi-inverse}.

Suppose now that $s>\frac 1p, s-\frac1p\notin\bbN_0$.  Again, let
$h\in\cB^p_a$, $\{h_n\}\subseteq \cT$, $h_n \to h$ in $B^p_a$ and set
$$
F_n(z)=\frac{1}{\sqrt{2\pi}}\int_{-a}^a\widehat {(h_n)_0}(\xi) |\xi|^{-s}
\big(e^{iz\xi}-P_m(iz\xi) \big)\, d\xi \,,
$$
where $m=\lfloor s-1/p\rfloor$.
Then, $F_n\in\cE_a$ and $\Delta^{\frac s2}
(F_n)_0= (h_n)_0$ by \cite[Corollary 3.4]{MPS-sob}.  Thus, $\{F_n\}$ is a Cauchy sequence in
$\Bsp$, that is, $\{(F_n)_0\}$ is a Cauchy sequence in
$E^{s,p}$, hence there exists a limit function $\widetilde F\in
E^{s,p}$. We need to prove that $\widetilde F$ is the restriction of some
entire function of exponential type $a$. 

Differentiating $m+1$ times, since $s':= m+1-s  \in (-1/p,1/p')$ the integrals below
converge absolutely so that
\begin{align}
  F^{(m+1)}_n(z)
  & =\frac{1}{\sqrt{2\pi}}\int_{-a}^a \widehat {(h_n)_0}(\xi) |\xi|^{-s}
(i\xi)^{m+1} e^{iz\xi}\, d\xi \notag \\
& =\int_{\bbR}
h_n(t)\cF^{-1} \big( (i\xi)^{m+1}|\xi|^{-s}e^{iz\xi}\chi_{[-a,a]}
\big) (t)\, dt \,. \label{RHS}
  \end{align}
 Then, if $-\frac 1p
<s'<0$,
the term on the
right hand side in \eqref{RHS} equals  
\begin{align*}
  \int_{\bbR}  h_n(t) R^{m+1} \cI_{s'}  \sinc (z-t)  \, dt  \,,
\end{align*}
and by Theorem \ref{riesz-potential} we obtain
\begin{align}
  |F^{(m+1)}_n(z)|
   & \le C \|(h_n)_0\|_{L^p}
\|  \sinc (z-\cdot) \|_{L^q} \notag \\
& \le C e^{a|y|}\|(h_n)_0\|_{L^p}\,, \label{unif-est}
\end{align}
where $\frac1q=\frac{1}{p'}-s'$, by the classical
Plancherel--P\'olya inequality.  If $0\le s'< \frac{1}{p'}$, we repeat
the same argument with $s'-1$ in place of $s'$, observing that the the term on the
right hand side in \eqref{RHS} equals
\begin{align*}
  \int_{\bbR}  h_n(t) R^m \cI_{s'-1} \big( D\sinc (z-\cdot) \big) (t)
  \, dt \,,
\end{align*}
and using the classical Bernstein inequality as well.

Therefore, the  convergence of
$\{h_n\}$ in $\cB^{p}_a$ implies the uniform convergence on compact subsets of
$\bbC$ of $\{ F^{(m+1)}_n\}$ and, in particular, the limit
function $G_{m+1}$ is of exponential type $a$.  Then, $F$ is the
anti-derivative of $G_{m+1}$ such that $P_{F;m,0}=0$ and
 $F|_\bbR=\widetilde F$, as we wished to
 show. 
This shows that the inverse is as in \eqref{Delta-s-mezzi-inverse} and
concludes the proof of the theorem.
\qed
\ms

\begin{cor}\label{unif-conv-comp-subsets}
Let $s>0$, $1<p<\infty$ and $s-\frac1p\not\in\bbN_0$.  Then, norm
convergence in $\cB^{s,p}_a$ implies uniform convergence on compact
subsets of $\bbC$.
  \end{cor}
  \proof
  Let $f \in \cB^{s,p}_a$.  From the identity, \eqref{Delta-s-mezzi-inverse}
with $h\in\cB^p_a$,
  arguing as in  \eqref{unif-est}, we obtain that
  $$
  |f^{(m+1)} (z)| \le e^{a|y|} \| \Delta^{\frac s2} f\|_{\cB^p_a} =
  e^{a|y|} \| f\|_{\cB^{s,p}_a} \,.
$$
Since $P_{f;m;0}=0$, it follows that for any compact $K\subseteq\bbC$,
$$
\sup_{z \in K} |f(z)|
\le C_K \| f\|_{\cB^{s,p}_a} \,.  \qed
$$

We are now ready to prove Theorems \ref{main-3} and \ref{main-4}.

  \proof[Proof of Theorem \ref{main-3}]
  We observe that the completeness follows from
the above corollary, or from
  the surjective isometry between $\Bsp$ and 
$\cB^p_a$.
For the second part of the theorem we argue as follows.
Let $h\in\cT$, and as in the proof of Theorem \ref{banach-isometry} 
define
$$
f(w) = (\Delta^{\frac s2})^{-1} h(w) 
= \frac{1}{\sqrt{2\pi}} \int_{-a}^a \widehat h_0(\xi) |\xi|^{-s} \big(
e^{iw\xi} - P_m(iw\xi)\big)\, d\xi \,, 
$$
and therefore
$$
P_{f(\cdot+iy);m;0}(w)=\frac{1}{\sqrt{2\pi}} \int_{-a}^a \widehat h_0(\xi) |\xi|^{-s} \big(e^{-y\xi}
P_m(iw\xi) - P_m(i(w+iy)\xi)\big)\, d\xi \,.
$$
Hence,
\begin{align*}
  F(w)
  & =f(w+iy)-P_{f(\cdot+iy);m;0}(w)
=\frac{1}{\sqrt{2\pi}} \int_{-a}^a \widehat h_0(\xi)e^{-y\xi} |\xi|^{-s} \big(
    e^{iw\xi} - P_m(i(w+iy)\xi)\big)\, d\xi \\
  & = \frac{1}{\sqrt{2\pi}} \int_{-a}^a \widehat h_0(\cdot +iy)(\xi) |\xi|^{-s} \big(
    e^{iw\xi} - P_m(iw\xi)\big)\, d\xi =  (\Delta^{\frac s2})^{-1}
    \big( h(\cdot+iy) \big) (w) 
    \,.
\end{align*}
Hence, from \eqref{Delta-s-mezzi-inverse}, we conclude that $F\in
\Bsp$ since, by the classical
Plancherel--P\'olya inequality (\cite{Young}), 
$h(\cdot+iy)\in\cB^p_a$ and by Theorem \ref{banach-isometry} we
obtain
\begin{align*}
 \|F\|_{\Bsp}=\|h(\cdot+iy)\|_{\cB^p_a}\leq e^{a|y|}\|h\|_{\cB^p_a}=e^{a|y|}\|f\|_{\Bsp}
\end{align*}
as we wished to show.  The conclusion now follows from Lemma
\ref{density-bernstein}. \qed

  \begin{proof}[Proof of Theorem \ref{main-4}]
If $f\in \Bsp$, then  $f_0\in E^{s,p}$ and $\supp\widehat
f_0\subseteq[-a,a]$, hence (i) implies (ii).  

If (ii) holds, then
$\Delta^{\frac s2} h \in L^p$ and $\supp \widehat{\Delta^{\frac s2}
  h} \subseteq \supp\widehat{  h} \subseteq [-a,a]$.  Hence,  it follows that 
$\Delta^{\frac s2} h=f_0$, for some $f\in \cB^p_a$. Setting 
$F=(\Delta^{\frac s2})^{-1}f$ we have $F\in \Bsp$. Hence,
$\Delta^{\frac s2} F_0= f_0 = \Delta^{\frac s2} h$.  Since $F_0,h\in
E^{s,p}$ and $\Delta^{\frac s2}$ is injective on 
$E^{s,p}$, it follows that $F_0=h$, that is, (ii) implies (i).

By applying the classical characterization of Bernstein spaces to
$\Delta^{\frac s2} h$ and Theorem \ref{banach-isometry} 
we easily see that (ii) and (iii) are equivalent.
\end{proof}
\ms

\section{Reconstruction formulas and sampling in
  $PW^s_a$}\label{sampling}

In this small section we make some comments and observations on reconstruction formulas
and sampling for the fractional Paley--Wiener spaces $PW^s_a$. In particular we conclude that the fractional Paley--Wiener spaces are not de Branges spaces.

\begin{prop}\label{shannon}
Let $0<s<\frac12$. Then, the set $\{\psi(\cdot-n\pi/ a)\}_{n\in\bbZ}$,
$$
\psi(z-n\pi/a)=\frac{1}{2\sqrt{a\pi}}\int_{-a}^a e^{in\frac\pi a \xi}e^{iz\xi}|\xi|^{-s}\, d\xi 
$$
is an orthonormal basis for $PW^s_a$.

If $s>\frac12, s-\frac12\notin\bbN$, the set $\{\psi(\cdot-n\pi/ a)\}_{n\in\bbZ}$,
$$
 \psi(z-n\pi/ a)=\frac{1}{2\sqrt{a\pi}}\int_{-a}^a e^{in\frac \pi
   a\xi}\big(e^{iz\xi}-P_m(iz\xi)\big)|\xi|^{-s}\, d\xi  
 $$
 is an orthonormal basis for $PW^s_a$.
\end{prop}
\begin{proof}
 It is a well known fact that the family of functions $\{\vp_n\}_{n\in\bbZ}$,
 $$
 \vp_n(z)=\frac{1}{2\sqrt{a\pi}}\int_{-a}^a e^{it(n\frac\pi
   a-z)}\, dt=\sqrt{a/\pi}\sinc \big(a(z-n\pi/ a)\big) 
 $$
 is an orthonormal basis for $PW_a$. The conclusion follows from Theorem \ref{hilbert-isometry}.
\end{proof}
 
We have the following consequences.
\begin{cor}
For every $f\in PW^s_a$ we have the orthogonal expansion
 \begin{equation}\label{shannon1}
  f(z)=\sum_{n\in\bbZ}\Delta^{\frac s2}f(n\pi/a)\psi(z-n\pi/ a) \,,
\end{equation}
where
the series converges in norm   and uniformly on compact subsets of
$\bbC$. Moreover, 
 \begin{equation}\label{shannon2}
  \|f\|^2_{PW^s_a}=\frac{a}{\pi}\sum_{n\in\bbZ}|\Delta^{\frac s2}f(n \pi /a)|^2.
 \end{equation}
\end{cor}

\proof
By the classical theory of Hilbert spaces and Plancherel's formula, we get 
\begin{align*}
 f&=\sum_{n\in\bbZ}\left<f \,|\, \psi(\cdot-n\pi/ a)\right>_{PW^s_a}\psi(\cdot-n\pi/ a)\\
 &=\sum_{n\in\bbZ} \frac{a}{\pi}\Big(\frac{1}{\sqrt{2\pi}}\int_{-a}^a
 |\xi|^s\widehat{f_0}(\xi) e^{in\frac{\pi}{a}\xi}\,
 d\xi\Big)\psi(\cdot-n\pi/ a)\\ 
 &=\sum_{n\in\bbZ}\Delta^{\frac s2}f(n\pi/a)\psi(\cdot-n\pi/ a),
\end{align*}
where  the series convergences in
$PW^s_a$-norm, hence  uniformly convergence on the compact
subsets of $\bbC$. Finally, formula \eqref{shannon2}
follows from Theorem \ref{hilbert-isometry} and the classical
Shannon--Kotelnikov formula.
\epf

A few comments are in order. Both the reconstruction formula
\eqref{shannon1} and the norm identity \eqref{shannon2} resemble some
known results for the Paley--Wiener space $PW_a$. In particular,
equation \eqref{shannon2} can be thought as a substitute of the
 Shannon--Kotelnikov  sampling theorem in the setting of fractional
Paley--Wiener spaces. However, these results are somehow
unsatisfactory: we recover the function $f$ and its norm from point
evaluations of the fractional Laplacian $\Delta^{\frac s2}f$  and not
of the function itself. Hence, it is a very natural question if we can
do something better. Indeed, this is the case for the reconstruction
formula \eqref{shannon1}, but we cannot really improve
\eqref{shannon2}. For simplicity, we now restrict ourselves to the
case $0<s<\frac12$. 

\begin{prop}
 Let $f\in PW^s_a$, $0<s<\frac12$. Then, for $z\in\bbC$,
 $$
 f(z)=\sum_{n\in\bbZ}f(n\pi/a)\sinc\big(a(z-n\pi/ a)\big) \,,
 $$
 where the series converges absolutely and uniformly on compact subsets of
 $\bbC$. 
\end{prop}

\proof
Let
$f\in\ PW^s_a$ and let $\{f_k\} \subseteq PW_a$ be a
sequence such that $f_k\to f$ in $PW^s_a$. Then, by the Shannon--Kotelnikov  theorem,
we have that
$$
f_k(x)=\sum_{n\in\bbZ} f_k(n\pi/a)\sinc\big(a(x-n\pi/ a)\big),
$$
where the series converges absolutely and in $PW_a$-norm. However,
norm convergence in $PW_a$ implies uniform convergence on compact
subsets of $\bbC$. Thus, we obtain
$$
f(x)=\lim_{k\to+\infty} f_k(x)=\lim_{k\to+\infty}\sum_{n\in\bbZ}
f_k(n\pi/ a)\sinc\big(a(x-n\pi/ a)\big)=\sum_{n\in\bbZ} f(n\pi/
a)\sinc\big(a(x-n\pi/ a)\big) \,. \qed
$$

We now point out that, in general, we cannot improve \eqref{shannon1}
with point evaluations of $f$ instead of its fractional
Laplacian. More generally, we would like to know if it is possible to
have a real sampling sequence in $PW^s_a$. We will see that, at least
in the case $0<s<\frac12$, this is not the case. 

\begin{defn}{\rm
 Let $\Lambda=\{\lambda_n\}_{n\in\bbZ}\subseteq\bbR$. The
 sequence $\Lambda$ is a sampling sequence for $PW^s_a$ if there
 exists two positive constants $A,B$ such that 
 $$
 A\|f\|^2_{PW^s_a}\leq
 \sum_{\lambda_n\in\Lambda}|f(\lambda_n)|^2
 =\sum_{\lambda_n\in\Lambda}|\left<f \,|\,
   K_{\lambda_n}\right>_{PW^s_a}|^2\leq B\|f\|^2_{PW^s_a} \,,
 $$
 where $K_{\lambda_n} $ is the reproducing kernel of $PW^s_a$, as in Corollary
 \ref{kernel-s-small}.
 }
\end{defn}

Since for $0<s<\frac12$ the space $PW^s_a$ can be identified with $L^2_a(|\xi|^{2s})$ via the
Fourier transform, the sequence $\Lambda=\{\lambda_n\}_{n\in\mathbb
  Z}$ is a sampling sequence for $PW^s_a$ if and only if the family of
functions $\{\widehat K_{\lambda_n}\}_{\lambda_n\in \Lambda}$ is a
frame for $L^2_a(|\xi|^{2s})$, that is, if and only if there exist two
positive constants $A,B$ such that 
$$
A\|\widehat{f_0}\|^2_{L^2_a(|\xi|^{2s})} \leq \sum_{\lambda_n\in
  Z}|\big<\widehat{f_0} \,|\,\widehat{
(K_{\lambda_n})_0}\big>_{L^2_a(|\xi|^{2s})}|^2\leq
B\|\widehat{f_0}\|^2_{L^2_a(|\xi|^{2s})}. 
$$

From Corollary \ref{kernel-s-small}, when $0<s<\frac12$ we  obtain that
$\widehat{
(K_{\lambda_n})_0} (\xi)= \frac{1}{\sqrt{2\pi}} e^{-i\bar
  \lambda_n\xi}|\xi|^{-2s}\chi_{[-a,a]}(\xi)$ and the following
result is easily proved.

\begin{prop}
The family $\{e^{-i \lambda_n\xi}|\xi|^{-2s}\}_{\lambda_n\in
  \Lambda}$ is a frame for $L^2_a(|\xi|^{2s})$ if and only if the
family $\{e^{-i \lambda_n\xi}|\xi|^{-s}\}_{\lambda_n\in \Lambda}$
is a frame for $L^2_a$. 
\end{prop}
\begin{proof}
 Assume that $\{e^{-i
   \lambda_n\xi}|\xi|^{-2s}\}_{\lambda_n\in\Lambda}$ is a frame for
 $L^2_a(|\xi|^{2s})$ and let $f$ be a function in $PW_a$. Then, 
\begin{align*}
  \|\widehat{ f_0}\|^2_{L^2_a}
  & =
 \||\xi|^{-s}\widehat{
 f_0}\|^2_{L^2_a(|\xi|^{2s})}\approx
 \sum_{\lambda_n\in\Lambda}\Big|\int_{-a}^a |\xi|^{-s}\widehat {f_0}(\xi)
 e^{i\lambda_n\xi}|\xi|^{-2s}|\xi|^{2s}\,
    d\xi\Big|^2  \\
  & =\sum_{\lambda_n\in\Lambda}\Big|\int_{-a}^a \widehat{
    f_0}(\xi) e^{i\lambda_n\xi}|\xi|^{-s}\Big|^2\, d\xi \,,
    \end{align*}
 hence, $\{e^{-i \lambda_n\xi}|\xi|^{-s}\}_{\lambda_n\in\Lambda}$
 is a frame for $L^2_a(|\xi|^{2s})$. The reverse implication is
 similarly proved. 
\end{proof}

Therefore, the sampling problem for $PW^s_a$, $0<s<\frac 12$, is
equivalent to study windowed frames for $L^2_a$. The following result,
due to C.-K. Lai \cite{Lai} (see also \cite{GL}) implies that we
cannot have real sampling sequences  for $PW^s_a$. Hence, we cannot
obtain an analogue of \eqref{shannon2} with point evaluations of the
function instead of point evaluations of its fractional Laplacian. 

\begin{Theorem}[\cite{Lai, GL}]
The family $\Big\{g(\xi)e^{i\lambda_n
  \xi}\Big\}_{\lambda_n\in\Lambda}$ is a frame for $L^2_a$ for some
sequence of points $\Lambda=\{\lambda_n\}_{n\in\bbZ}\subseteq
\bbR$ if and only if there exist positive constants $m,M$ such that
$m\leq g(\xi)\leq M$. 
\end{Theorem}

We conclude the section with one last comment about fractional
Paley--Wiener spaces and de Brange spaces. These latter spaces were
introduced by L. de Branges \cite{dB} and have been extensively
studied in the last years. Among others, we recall the papers
\cite{OS, BBB2, BBB, BBP, BBH, ABB}. The space $PW_a$ is the model
example of a de Branges space. A classical result, see \cite{dB},
states that de Brange spaces always admit a real sampling sequence,
or, equivalently, always admit a Fourier frame of reproducing
kernels. The above discussion proves that this is not the case for the
spaces $PW^s_a$, $0<s<\frac12$. Therefore, the following result
holds. 
\begin{thm}
 The fractional Paley--Wiener spaces $PW^s_a$, $0<s<\frac12$, are not de Branges spaces. 
\end{thm}
\ms

\section{Boundedness of the orthogonal projection}\label{bdd-projection}

In the previous sections we proved that the spaces $PW^s_a$ can be equivalently described as
$$
PW^s_a=\left\{f\in E^{s,2} : \supp\widehat f\subseteq [-a,a]\right\},
$$
thus, it is clear that the spaces $PW^s_a$ are closed subspaces of the
Hilbert spaces $E^{s,2}$. Therefore, we can consider the Hilbert space
projection operator $\bP_s:E^{s,2}\to PW^s_a$. 

In the case $0<s<\frac 12$ we get from Theorem \ref{main-1} that the
projection operator $\bP_s$ is explicitly given by the formula 
\begin{equation}\label{proj-1}
\bP_sf(x)=\frac{1}{\sqrt{2\pi}}\int_{\bbR} \widehat
f(\xi)\chi_{[-a,a]}(\xi) e^{ix\xi}\, d\xi \,.
\end{equation}
Similarly, we explicitly deduce $\bP_s$ in the case $s>\frac12$ from Theorem \ref{main-2}, 
\begin{equation}\label{proj-2}
\bP_sf(x)=\frac 1{2\pi}\int_\bbR \widehat f(\xi)\chi_{[-a,a]}(\xi) 
\big(e^{ix\xi}-P_m(ix\xi) \big)\, d\xi \,.
\end{equation}

A very natural question is to investigate whether the operator $\bP_s$
densely defined on $E^{s,p}\cap E^{s,2}$ extends to a bounded operator
$\bP_s: E^{s,p}\to\Bsp$ assuming that $s-\frac12\notin\bbN_0, s-\frac1p\notin\bbN_0$ and $\lfloor s-\frac12\rfloor=\lfloor s-\frac1p\rfloor$. This is the content of Theorem \ref{main-5}
which we now prove.

\proof[Proof of Theorem \ref{main-5}]
We first assume $0<s<\frac12$. Let $f$ be a function in $E^{s,p}\cap
E^{s,2}$. By definition of $E^{s,p}\cap E^{s,2}$, we can assume $f$ to
be in the Schwartz space $\cS$.  
Then, the projection $\bP_sf$ is given by \eqref{proj-1}. The function
$\bP_sf$ clearly extends to an entire function of exponential type
$a$, which we still denote by $\bP_sf$. Moreover, we assumed
$f\in\cS$, so that, for instance, $\bP_sf$ is a well-defined $L^2$
function with a well-defined Fourier transform. Thus, 
$$
\Delta^{\frac
  s2}\bP_sf(x)=\cF^{-1}\big(|\cdot|^{s}\chi_{[-a,a]}\widehat{f}\,\big)
=\cF^{-1}\big(\chi_{[-a,a]}
\widehat{\Delta^{\frac
    s2}f}\big). 
$$
Hence
$$
\|\bP_sf\|_{\cB^{s,p}_a} =\|\Delta^{\frac s2}\bP_sf\|_{L^p(\bbR)}
\leq C \|\Delta^{\frac s2}f\|_{L^p(\bbR)} = C\|f\|_{E^{s,p}},
$$
where the inequality holds since $\chi_{[-a,a]}$ is an $L^p$-Fourier
multiplier for any $1<p<+\infty$.  Therefore, $\bP_s$ extends to a
bounded operator $\bP_s: E^{s,p}\to\Bsp$ when $0<s<\frac 12$. 
%

Assume now $s>\frac 12$. Then, given $f\in E^{s,p}\cap
E^{s,2}\cap\cS$, the projection $\bP_sf$ is given by \eqref{proj-2},
that is, 
\begin{align*}
\bP_sf(x)&=\frac 1{2\pi}\int_\bbR \widehat f(\xi)\chi_{[-a,a]}(\xi) (e^{ix\xi}-P_m(ix\xi))\, d\xi\\
&=\frac 1{2\pi}\int_\bbR \widehat f(\xi)\chi_{[-a,a]}(\xi)e^{ix\xi}\,
d\xi-\frac 1{2\pi}\int_\bbR \widehat f(\xi)\chi_{[-a,a]}(\xi)
P_m(ix\xi)\, d\xi\\ 
&=: (\bP_sf)^1(x)+(\bP_sf)^2(x).
\end{align*}
As before, $(\bP_sf)^1$ is a well-defined $L^2$ function with a
well-defined Fourier transform, whereas $(\bP_sf)^2$ is a polynomial
of degree $m=\lfloor s-1/2\rfloor <s$, thus its fractional Laplacian $\Delta^{\frac
  s2}$ is zero. Therefore,  
$$
\Delta^{\frac s2}\bP_sf(x)=\frac{1}{\sqrt{2\pi}}\int_{\bbR}|\xi|^{s}\widehat
f(\xi)\chi_{[-a,a]}(\xi) e^{ix\xi}\, d\xi=
\cF^{-1}\big(\chi_{[-a,a]}\widehat {\Delta^{\frac s2}f}\big)(x). 
$$
Once again we have
$$
\|\bP_sf\|_{\cB^{s,p}_a}=\|\Delta^{\frac s2}\bP_sf\|_{L^p(\bbR)}\leq C
\|\Delta^{\frac s2}f\|_{L^p(\bbR)} = C\|f\|_{E^{s,p}}, 
$$
since $\chi_{[-a,a]}$ is a $L^p$-Fourier multiplier for any
$1<p<\infty$.
\epf

 \section{Proof of Lemma  \ref{density-bernstein}} \label{density-sect}

\proof[Proof of Lemma  \ref{density-bernstein}]

We recall that, given a
function $\vp$ on $\bbR$, for $t>0$ we set $\vp_t=\frac1t 
\vp(\cdot/t)$.  We also set $\vp^t =\vp(t\cdot)$ and observe that
$\cF(\vp_t) = (\cF\vp)^t$.  Moreover, it is easy to see that for all
$p\in[1,\infty)$,
$\vp_r, \vp^r \to \vp$ in $L^p$, as $r\to 1$.

We first claim that the subspace $\bigcup_{\delta>0}\{ f \in \cB^p_{a-\delta}:\, f_0\in\cS\}$ is
dense in $\cB^p_a$, $1<p<\infty$.  
Let $f \in \cB^p_a$ be given. Then $\supp \widehat{f_0}
\subseteq[-a,a]$ and if $0<r<1$, $f^r \in \cB^p_{ar}$ so that
$\supp \widehat{f_0^r} \subseteq[-ar,ar]$.  Let $\delta=(1-r)a/2$ and let
$\vp\in C^\infty_c[-1,1]$, $\vp=1$ on $[-\frac14,\frac14]$, $\int
\vp=1$. Then $\widehat{f_0^r}*\vp_\delta \in C^\infty_c$ and
$\supp \widehat{f_0^r}*\vp_\delta \subseteq [-a+\delta,a-\delta]$.  
Therefore, $\cF^{-1} \big(\widehat{f_0^r}*\vp_\delta\big)\in\cS$ extends to
a function $f_{(\delta)} \in
\cB^p_{a-\delta}$ and  $f_{(\delta)} \to f$ in $\cB^p_a$ as
$\delta\to0$.   This proves the claim.

Next, let $f\in \cB^p_{a-\delta}$ be such $f_0\in\cS$, for
$\delta>0$.  Let $\eta_{(\delta)}\in C^\infty_c[-\delta,\delta]$, $\eta=1$ on
$[-\delta/2,\delta/2]$. Then $(1-\eta_{(\delta)})\widehat{f_0} \in C^\infty_c$
and has support in $\{\xi:\, \delta/2\le|\xi|\le a\}$.  Therefore,
$\cF^{-1} \big( (1-\eta_{(\delta)}) \widehat{f_0} \big) \in\cS_\infty$ and
extends to a function in $ \cB^p_a$.  Thus, it suffices to show that
$\| \cF^{-1} \big( \eta_{(\delta)}\widehat{f_0} \big) \|_{L^p}\to 0$
as $\delta\to0$.  This fact follows by observing that we may choose
$$
\eta_{(\delta)}  = (\chi*\vp)^{1/\delta}
$$
where $\vp\in C^\infty_c[-\frac12,\frac12$ with $\int \vp=1$.  Then,
it is clear that 
$\eta_{(\delta)} \in  C^\infty_c[-\delta,\delta]$, and $\eta=1$ on
$[-\delta/2,\delta/2]$. Finally, for $q \in(1,\infty)$, it is easy to
see that
$$
\| \cF^{-1} \eta_{(\delta)} \|_{L^q} = \delta^{1-1/q}  \|
\widehat{\chi}\widehat{\vp}\|_{L^q} \to 0 
$$
as $\delta\to0$.
\qed

\begin{cor}\label{density-PWs-1}
  Let $s>0$, $p\in(1,\infty)$, $s-\frac1p\notin\bbN_0$, and set
  $m=\lfloor s-\frac1p\rfloor$. For $s>\frac1p$, 
  set
  $\cT_m = \big\{f\in\cE_a:f_0\in\cS_\infty, P_{f;m;0}=0 \big\}$.
 Then,   if $0<s<\frac1p$
the subspace $\cT$ is dense in $\cB_a^{s,p}$, whereas if $s>\frac1p$
the subspace $ \cT_m$ is
dense in $\cB_a^{s,p}$ if $s>\frac1p$, $s-\frac1p\notin\bbN_0$. 
\end{cor}

\proof
We only prove the case $s>\frac1p$, $s-\frac1p\notin\bbN_0$, the other
case being easier. By Lemma \ref{density-bernstein} and Theorem
\ref{banach-isometry} we 
have that
$(\Delta^{\frac s2})^{-1}\big( \cT\big)$ is dense in $\cB_a^{s,p}$.
Thus, it suffices to show that this latter space is contained in
$\cT_m$.  Let $h\in\cT$ and let $f=\Delta^{-\frac s2} h$ be given by
\eqref{Delta-s-mezzi-inverse}.  It is clear that $P_{f;m;0}=0$.  Moreover,
 $f^{(m+1)}_0 = \cF^{-1}\big( (i\xi)^{m+1} |\xi|^{-s}
\widehat{h_0}\big) \in \cS_\infty$ and extends to a function in
$\cB^{s,p}_a$, by Theorem \ref{main-4}.  This easily implies that $f_0
\in\cS_\infty$ and the conclusion follows.
\qed

\section{Final remarks and open questions}\label{Final-Sect}
We believe that the fractional spaces we introduced are worth
investigating and, as we mentioned, they arise naturally in a several
variable setting (\cite{MPS}). 

We mention a few questions that remain open.  First of all, it
is certainly of interest to consider the cases $s-\frac1p \in\bbN_0$.
As we pointed out already, these cases correspond to the critical
cases in the Sobolev embedding theorem.  As shown by
Bourdaud, when $s-\frac1p \not\in\bbN_0$,
the realization spaces $E^{s,p}$ of $\dot{W}^{s,p}$ are the
unique realization spaces whose norms are homogeneous with respect the
natural dilations.  On the other hand, when $s-\frac1p \in\bbN_0$
there exists no realization space of  $\dot{W}^{s,p}$ whose norm is
homogeneous.  In these case, it would be natural to define the
realization space as the interpolating space between two spaces with  
$s-\frac1p \not\in\bbN_0$. Thus, a natural definition may be
\begin{equation*}
\Bsp=\big\{f\in \cE_{a}:\,  [f_0]_m\in \dot{W}^{s,p} \text{ and if } m\ge
1/p\,, P_{f_0;m;0}=0  \big\}\,,
\end{equation*}
where $[f_0]_m$ denotes the equivalence class of $f_0$ in
$\cS'/\cP_m$.  In any event, these spaces remain to be investigated.
Naturally, another question that remains open is the boundedness of
the orthogonal projection $\bP:E^{s,p}\to \cB^{s,p}_a$ in the cases  $s-\frac1p \in\bbN_0$.
 Such boundedness would allow one to explicitly describe the dual
 space of $\cB^{s,p}_a$, for the whole scale $s>0$ and
 $p\in(1,\infty)$.

 The Paley--Wiener space is a very special instance of a de Branges
 spaces. These spaces where introduced by de Branges also in
 connection with the analysis of
 the canonical systems, see e.g. \cite{dB,Romanov}. 
 It would be interesting to determine whether the fractional
 Paley--Wiener spaces $PW^s_a$ also arise to the solution of a
 canonical system defined in terms of the fractional derivative.

In \cite{BBH} it is shown that the Paley--Wiener space, and, more
generally, any de Branges space, coincides as set with a Fock-type
space with non-radial weight. The Paley--Wiener (or de Branges) norm
given by an integral on the real line is replaced by an equivalent
weighted integral on the complex plane. We wonder if an analogous
result holds true for the fractional Paley--Wiener spaces. 
 
Another important fact about the classical Paley--Wiener space is
that, up to a multiplication by an inner function, it admits a
representation as a \emph{model space} of $H^2(\bbC_+)$, the Hardy
space of the upper half-plane. We recall that a model subspace of
$H^2(\bbC_+)$ is defined as $K_{\Theta}=H^2(\bbC_+)\ominus \Theta
H^2(\bbC_+)$ where $\Theta$ is an inner function in $\bbC_+$. 
It would certainly be interesting to investigate the analogous spaces
appearing in the case of the
fractional
Paley--Wiener and Bernstein spaces.

\bibliography{PW-bib-9-2}
\bibliographystyle{amsalpha}
\vspace{-.1cm}
\end{document}